\documentclass[11pt]{amsart} % Use the amsart document class

\usepackage[utf8]{inputenc}
\usepackage[T1]{fontenc}
\usepackage{amsmath}
\usepackage{amsthm}
\usepackage{amsfonts}
\usepackage{amssymb}
\usepackage{verbatim}
\usepackage{graphicx}
\usepackage{dsfont}
\usepackage{url}
\usepackage{enumerate}
\usepackage[pdftex,bookmarks=true]{hyperref}
\usepackage[normalem]{ulem}
\usepackage[usenames, dvipsnames]{color}
\usepackage{stmaryrd}

\usepackage{xcolor}                   % color for links
\definecolor{green}{HTML}{2ECC71}
\definecolor{blue}{HTML}{3498DB}
\definecolor{red}{HTML}{E74C3C}
\definecolor{orange}{HTML}{FD6A02}

\hypersetup{
	colorlinks=true,
	linkcolor=black,
	citecolor=blue,
	filecolor=magenta,      
	urlcolor=cyan,
	pdfpagemode=FullScreen,
}

\theoremstyle{plain}
\newtheorem{theorem}{Theorem}[section]

\newtheorem{prop}[theorem]{Proposition}

\newtheorem{lemma}[theorem]{Lemma}
\newtheorem{lm}[theorem]{Lemma}
\newtheorem{corollary}[theorem]{Corollary}
\newtheorem{cor}[theorem]{Corollary}

\newtheorem{thm}[theorem]{Theorem}

\theoremstyle{definition}

\newtheorem{remark}[theorem]{Remark}

\begin{document}
	
	\title[Random polynomials with dependent coefficients]{Roots of random trigonometric polynomials with general dependent coefficients}
	
	\author[J. Angst]{Jürgen Angst}
	\address[J. Angst]{IRMAR, UMR CNRS 6625, Université de Rennes, F-35000 Rennes, France}
	\email{jurgen.angst@univ-rennes.fr}
	
	\author[O. Nguyen]{Oanh Nguyen}
	\address[O. Nguyen]{Division of Applied Mathematics, Brown University, Providence, RI 02906, USA}
	\email{oanh\_nguyen1@brown.edu}
	
	\author[G. Poly]{Guillaume Poly}
	\address[G. Poly]{IRMAR, UMR CNRS 6625, Université de Rennes, F-35000 Rennes, France}
	\email{guillaume.poly@univ-rennes.fr}
	
	\thanks{Nguyen is supported by NSF grant DMS–2246575 and the Salomon grant.}
	
	
	\begin{abstract}
We consider random trigonometric polynomials with general dependent coefficients. We show that under mild hypotheses on the structure of dependence, the asymptotics as the degree goes to infinity of the expected number of real zeros coincides with the independent case. To the best of our knowledge, this universality result is the first obtained in a non-Gaussian dependent context. Our proof highlights the robustness of real zeros, even in the presence of dependencies. These findings bring the behavior of random polynomials closer to real-world models, where dependencies between coefficients are common.
	\end{abstract}
	
	\maketitle

%\begin{keyword}[class=MSC]
%\kwd[Primary ]{26C10}
%\kwd[; secondary ]{30C15 , 42A05 , 60F17 , 60G55}
%\end{keyword}

%\setcounter{tocdepth}{1}
%\tableofcontents
%\par
%\vspace{-1cm}
\section{Introduction and main results}

\subsection{Random trigonometric polynomials and universality}
The study of the number of real zeros of random algebraic or trigonometric polynomials is the object of a vast literature. Of particular interest is the question of the universality of the large degree asymptotics of this number of zeros, which consists of determining whether the latter asymptotic behavior depends on the specific choice of the joint distribution of the random coefficients.
\par
\medskip
The case of random trigonometric polynomials whose coefficients are independent and identically distributed has been studied intensively. In this setting, the asymptotics of the expected number of real zeros is known to display an universal behavior, at both local and global scales, as established in a series of papers, see for example \cite{MR3563891,MR3718101,MR3846831} and recently \cite{MR4340724} which provides the most general conditions. Remarkably, even if the asymptotics of the expected number of zeros is universal, its variance is not, as established in \cite{MR3980307,MR4452640}.
\par
\medskip
While the case of independent and identically distributed coefficients has been thoroughly studied, real-world phenomena often involve correlations between variables. In the context of random trigonometric polynomials, considering dependent coefficients introduces a richer and more complex setting. This raises natural questions about the asymptotic behavior of the number of real zeros and whether the universality observed in the independent case extends to the dependent setting. In particular, it becomes important to understand how the introduction of correlations affects the expected number of zeros and its variance, and whether new phenomena emerge when moving from independence to dependence. This extension not only broadens the scope of existing results but also brings the study of random polynomials closer to models encountered in practice, where dependencies between coefficients are common.
\par
\medskip
In the case where the random coefficients are correlated, the situation becomes more intricate,  and all the results obtained so far deal with the Gaussian case, namely when the random coefficients form a stationary Gaussian sequence with a given correlation function $\rho$, see \cite{MR3876743, MR4491425} and the references therein. In this case, the universality of the asymptotic of the expected number of zeros is related to fact that the associated spectral measure  admits an absolutely continuous component and in particular to the Lebesgue measure of its support. In this Gaussian setting, a key tool is the celebrated Kac--Rice formula which allows to express the expected number of zeros as an explicit integral of correlation functions. 
\par
\medskip
To our knowledge, the study of the number of zeros of random trigonometric polynomials with general dependent coefficients is completely open and in this article, we first focus on the case of random polynomials with $m-$dependent coefficients, and then illustrate how to extend our results to long-range dependency in Theorem \ref{theo.main.gauss}. We note that our method is quite robust and the results can be extended to a broader generality which we do not pursue here.

\subsection{General model and notations}\label{assumption}
We are thus interested in the nodal sets associated with random trigonometric polynomials of the form
\[
\sum_{k=1}^n a_{k,n}\cos\left(k t\right)+b_{k,n}\sin\left(kt \right), \quad t \in [0,2\pi],
\]
where the coefficients $a_{k,n}$ and $b_{k,n}$ are real, dependent random variables. To be more precise, let us consider an abstract probability space $(\Omega, \mathcal F, \mathbb P)$. A generic element in $\Omega$ will be denoted by $\omega \in \Omega$, and $\mathbb E$ will denote the expectation with respect to $\mathbb P$. We suppose that the latter probability space carries two independent copies of random arrays $(a_{k,n})_{k \geq 1, n \geq 1}$ and $(b_{k,n})_{k \geq 1, n \geq 1}$ satisfying the following conditions \par
\medskip
\begin{enumerate}
\item[{\bf (A.1)}] For all $n\geq 1, k \geq 1$, we have $\mathbb E[a_{k,n}]=0$, $\mathbb E[a_{k,n}^2]=1$ and there exists $\eta>0$ such that
\[
K_{\eta}:=\sup_{k\geq 1,n\geq 1} \mathbb E[|a_{k,n}|^{2+\eta}]<+\infty.
\]
\item[{\bf(A.2)}] For all $n \geq 1$, the sequences $(a_{k,n})_{k \geq 1}$ and $(b_{k,n})_{k \geq 1}$ are weakly stationary, with common associated covariance function $\rho_n$
\[
\rho_n(|k-\ell|):=\mathbb E[ a_{k,n} a_{\ell,n}]=\mathbb E[ b_{k,n} b_{\ell,n}], \quad \text{for} \;\;   k,\ell\geq 1.
\]
\end{enumerate}
Thanks to Bochner--Herglotz Theorem, each correlation function $\rho_n$ is then associated to a unique symmetric probability measure $\mu_{\rho_n}$ on the circle, the so-called spectral measure
\[
\rho_n(k) = \frac{1}{2\pi} \int_{-\pi}^{\pi} e^{i kx} \mu_{\rho_n}(dx).
\]
\begin{enumerate}
\item[{\bf (A.3)}] 
In all the sequel, we will assume that for each $n\geq 1$, the measure $\mu_{\rho_n}$ admits a density $\psi_{\rho_n}$ with respect to the Lebesgue measure on $[-\pi,\pi]$ and that there exists a density $\psi_{\rho}$  such that as $n$ goes to infinity
\begin{equation}\label{eq.lowerb}
\limsup_{n \to +\infty}||\psi_{\rho_n}-\psi_{\rho}||_{\infty} = 0, \quad \text{with}\quad \kappa_{\rho}:=\inf\{ \psi_{\rho}(x), x \in [-\pi, \pi]\}>0.
\end{equation}
\end{enumerate}

\subsubsection{The $m_n-$dependent case}\label{sec.mod.mdep}
First, of particular interest is the case where the correlation functions $\rho_n$ have finite support, which corresponds to \newline$m_n-$dependent sequences. Recall that a sequence $(a_k)_{k \geq 1}$ is said to be $m-$dependent if, for any given two subsets $I$ and $J$ of $\mathbb N^*$,  the families $(a_{i})_{i \in I}$ and $(a_{j})_{j \in J}$ are independent as soon as $d(I,J):=\inf\{|i-j|, i \in I, j \in J\}> m$. In this first case, the sequence of spectral densities $(\psi_{\rho_n})$ is thus a sequence of trigonometric polynomials of degrees $2m_n$, namely
\[
\psi_{\rho_n}(x) = \sum_{k=-m_n}^{m_n} \rho_{n}(|k|) e^{i kx} 
\]
and Condition (A.3) above requires this sequence to converge uniformly to a positive limit. 

\subsubsection{Functional of a Gaussian sequence}\label{sec.mod.gauss}
Secondly, we will consider the case where the random coefficients are functions of an auxiliary stationary Gaussian sequence, namely the case where we have a representation of the form $a_{k,n}=a_k=H(X_k)$ where 
\begin{itemize}
\item $(X_k)_{k \geq 1}$ is a stationary Gaussian process with its own correlation function $\rho_G$;
\item  $H$ is a measurable function admitting a finite $(2+\eta)$ moment with respect to the standard Gaussian measure.
\end{itemize}
If one assume that the covariance $\rho_G$ corresponds to a spectral density $\psi_G$ on $[-\pi, \pi]$. Then, if $H(x)=\sum_{q} c_q H_q(x)$ is the Hermite decomposition of $H$, we have indeed
\[
\psi_{\rho}(x) = \sum_q c_q^2 q! \psi_{G}^{\ast q}(x),
\]
where $\psi_{G}^{\ast q}(x)=\psi_G \ast \ldots \ast \psi_G(x)$ is the convolution of $\psi_G$ with itself $q-$times. In particular, if $\psi_G$ is lower bounded by a positive constant on a small interval, then for $q$ sufficiently large, $\psi_{G}^{\ast q}$ is uniformly lower bounded on the whole period $[-\pi, \pi]$ so that $\psi_{\rho}$ is also lower bounded by a positive constant and  Condition (A.3) is satisfied.\par
\medskip
As an illustrative example, the reader can keep in mind the case where $a_k=\text{sign}(X_k)$, where the relation between the two correlation functions $\rho$ and $\rho_G$ is explicit since
\[
\begin{array}{ll}
\rho(|k-\ell|) & =\mathbb E[ a_k a_{\ell}] = (+1)\times \mathbb P( X_k X_{\ell}>0) +(-1)\times \mathbb P( X_k X_{\ell}<0)\\
\\
& =\left( \frac{1}{2} + \frac{1}{\pi} \arcsin(\rho_G(|k-l|)) \right) - \left( \frac{1}{2} - \frac{1}{\pi} \arcsin(\rho_G(|k-l|)) \right) \\
\\
& =\frac{2}{\pi} \arcsin(\rho_G(|k-l|)).
\end{array}
\]

\subsection{Main results}

To the random arrays $(a_{k,n})_{k \geq 1, n \geq 1}$
and $(b_{k,n})_{k \geq 1, n \geq 1}$, we associate a sequence of random trigonometric polynomials setting for all $n \geq 1$
\[
f_n(t):=\sum_{k=1}^n a_{k,n}\cos\left(k t\right)+b_{k,n}\sin\left(kt \right).
\]
In the sequel, we will focus on the asymptotics of the number of zeros of $f_n$ in $[0, 2\pi]$ denoted by
\[
\mathcal N(f_n, [0,2\pi]) := \# \{ t \in [0,2\pi], \; f_n(t)=0\}.
\]
Our first main result shows that the expected number of zeros of the random trigonometric polynomials with $m_n-$dependent coefficients obeys the same asymptotics as in the independent case, provided the growth of $m_n$ is controled.
 \begin{theorem}\label{theo.main.mdep}
There exists a positive exponent $\gamma_0=\gamma_0(\eta)$ such that if $(a_{k,n})_{k \geq 1, n \geq 1}$ and $(b_{k,n})_{k \geq 1, n \geq 1}$ are two independent arrays of $m_n-$dependent  random variables satisfying Conditions (A.1)--(A.3), with $m_n = O(n^\gamma)$ and $\gamma<\gamma_0$, then the expected number of zeros of the associated random trigonometric polynomial $f_n$ satisfies the universal asymptotics
\[
\frac{\mathbb E\left[ \mathcal N(f_n, [0,2\pi])\right]}{n} \xrightarrow[n \to +\infty]{} \frac{2}{\sqrt{3}}.
\]
 \end{theorem}
 
 \begin{remark}
The proof below shows that, if the random coefficients have a finite $(2+\eta)$ moment,  the exponent $\gamma_0$ can be chosen as $\gamma_0(\eta)=\frac{\eta}{2(5+7\eta)}$ but the latter threshold is not meant to be sharp. The statement of Theorem \ref{theo.clt} below suggests that the optimal value should be closer to $\frac{\eta}{2(1+\eta)}$. In particular, if the coefficients have uniformly bounded moments of all orders, letting $\eta$ go to infinity, the threshold should be closer to $1/2$.
 \end{remark}
 
Since the threshold of dependence $m_n$ is allowed to go to infinity with $n$, via some approximation arguments, it is possible to deduce from Theorem \ref{theo.main.mdep} analogue universality statements for sequences of coefficients with an arbitrarily long range of dependency. 
For example, the next theorem establishes the universality of the expected number of zeros in the case where the random coefficients are functionals of a weakly stationary Gaussian sequence with fast decorrelation.

\begin{theorem}\label{theo.main.gauss}
Let us consider $(a_k)_{k \geq 1}$ and $(b_k)_{k \geq 1}$ two independent copies of a process of the form $a_k=H(X_k)$ where $(X_k)$ is a stationary Gaussian process with correlation function $\rho_G \in \ell^1(\mathbb N)$. We suppose that the associated spectral density $\psi_G$ is $\mathcal C^{\infty}$ and positive on the whole period and that the function $H$ has a finite $(2+\eta)$ moment with respect to the standard Gaussian measure. Then, asymptotically, the expected number of zeros of the associated random trigonometric polynomial $f_n$ is universal 
\[
\frac{\mathbb E\left[ \mathcal N(f_n, [0,2\pi])\right]}{n} \xrightarrow[n \to +\infty]{} \frac{2}{\sqrt{3}}.
\]
 \end{theorem}

\begin{remark}
A representative example of a Gaussian process satisfying the hypotheses of the above Theorem \ref{theo.main.gauss} is the discrete version of Bargman--Fock process, associated with the covariance function $\rho_G(k)=e^{-k^2/2}$. In that case, we have indeed 
\[
e^{-k^2/2} =\hat{f}_G(k)=\int_{\mathbb R} e^{- i kx } f_G(x)dx, \quad \text{with} \quad f_G(x):=\frac{e^{-x^2/2}}{\sqrt{2\pi}},
\]
so that by Poisson summation formula, the spectral density is given by 
\[
\psi_G(x)=\sum_{k \in \mathbb Z} \hat{f}_G(k) e^{i k x} = 2\pi \sum_{k \in \mathbb Z}f_G(x+2\pi k) = \sqrt{2\pi} \sum_{k \in \mathbb Z} e^{-\frac{1}{2}(x+2\pi k)^2}.
\]
In particular, we have 
\[
\kappa_G:=\inf_{x \in [-\pi,\pi]} \psi_G(x)\geq  \sqrt{2\pi} e^{-\frac{\pi^2}{2}}.
\]
In the same manner, for $\rho_G(k)=e^{-|k|}$, we have 
\[
e^{-|k|} =\hat{g}_G(k)=\int_{\mathbb R} e^{- i kx } g_G(x)dx, \quad \text{with} \quad g_G(x):=\frac{1}{\pi(1+x^2)},
\]
so that this time the spectral density is given by 
\[
\psi_G(x)=\sum_{k \in \mathbb Z} \hat{g}_G(k) e^{i k x} = 2\pi \sum_{k \in \mathbb Z}g_G(x+2\pi k) = \sum_{k \in \mathbb Z} \frac{2}{1+(x+2k\pi)^2}>\frac{2}{1+\pi^2}.
\]
\end{remark}

\begin{remark}
As already mentioned above, a prototypical choice of function $H$ in Theorem \ref{theo.main.gauss} would be $H(X_k)=\text{sign}(X_k)$. This simple choice yields a natural model for random Littlewood polynomials with dependent coefficients, for which our universality statement applies.
\end{remark}

\begin{remark}
The proof of Theorem \ref{theo.main.gauss} actually shows that its conclusion holds as soon as we have $\psi_G$ is of class $\mathcal C^{D_0}$ for $D_0=D_0(\eta)$ sufficiently large, which can be expressed as a function of the threshold $\gamma_0=\gamma_0(\eta)$, namely $D_0(\eta)=57/2+20/\eta$. As above, the index of regularity $D_0$ is not meant to be sharp. 
\end{remark}

 \subsection{Global strategy of the proof} 
At a high level, the global strategy behind the proof of Theorem \ref{theo.main.mdep} is similar to the one adopted in the references in \cite{MR4350983,MR4491425}. It consists in showing that, when properly localized in normalized, the random field $(f_n(t))$ admits a universal scaling limit in distribution, and that the convergence holds in the $\mathcal C^1$ topology, which ensure that the nodal set associated with $f_n$ converges in distribution towards the nodal set of the limit. To conclude that the expected number of zeros then converges to a universal limit, one is left to establish anti-concentration estimates for the field which guaranty that the number of zeros is uniformly integrable. 
\par
\medskip
In comparison with other works on random trigonometric polynomials with dependent coefficients, one of the main difficulties is that we are not in a Gaussian context here, hence the use of the celebrated Kac--Rice formula is prohibited, the expected number of zeros can not be written as a deterministic integral involving the correlations functions of the field and its derivatives. In the same way, since we are not in a Gaussian context, the law of the field is not characterized by the covariance function. 
\par
\medskip
Another major difficulty comes from the fact that we are dealing with general coefficients, possibly discrete, for which anti-concentration estimates are notoriously difficult to establish and are known to be weaker than for their Gaussian or continuous analogues. Due to the non-independent, non-Gaussian framework, the characteristic function $S_n$ is here out of reach and in particular, we cannot use the standard tools to establish small-ball estimates such as Hal\'asz method, Esseen inequality, or the recently developed inverse Littlewood-Offord theory (we refer to the excellent lecture notes by Krishnapur \cite{krishnapur2016anti} for an overview of these methods). Nevertheless, we come up with an elegant solution that incorporate the classical analytic tool of Tur\'an's lemma and combinatorial arguments, the latter being motivated by those of Erd\H{o}s \cite{erdos1945lemma}.
\par
\medskip
Let us give a few more details on each step of the proof and on the plan of the rest of the paper. Given the sequence of random coefficients  $(a_{k,n},b_{k,n})_{k \geq 1, n \geq 1}$ defined on the probability space $(\Omega, \mathcal F,\mathbb P)$, we consider an independent random variable $X$ with uniform distribution $\mathbb P_X=\frac{dx}{2\pi}$ in $[0,2\pi]$. This can be achieved by considering the product space 
\[
(\Omega \times [0,2\pi], \mathcal F \times \mathcal B([0,2\pi]), \mathbb P \otimes \mathbb P_X),
\]
with $X$ seen as the identity map from $([0,2\pi],\mathcal B([0,2\pi])$ to itself. 
The expectation with respect to $\mathbb P_X$ will be denoted by $\mathbb E_X$ in the sequel.
The rescaled version of the initial field $(f_n(t))$ that we will consider in the sequel is defined as
\[
S_n(t):=\frac{1}{\sqrt{n}}\sum_{k=1}^n a_{k,n}\cos\left(k X+\frac{kt}{n}\right)+b_{k,n}\sin\left(k X+\frac{kt}{n}\right)=\frac{1}{\sqrt{n}} f_n\left(X+\frac{t}{n}\right).
\]
In other words, $S_n$ is a normalized, zoomed-in version of $f_n$ in a window of size $2\pi/n$ based at the random point $X$.
Following Lemma 3 of \cite{MR4350983}, the number of zeros $\mathcal N(f_n, [0,2\pi])$ of $f_n$ can then be rewritten as
\[
\frac{\mathcal N(f_n, [0,2\pi])}{n} = \mathbb E_X \left[  \mathcal N\left(f_n, \left[X,X+\frac{2\pi}{n}\right]\right)\right] = \mathbb E_X \left[  \mathcal N\left(S_n, \left[0,2\pi\right]\right)\right],
\]
and we have thus the  following representation of the expected number of zeros 
\begin{equation}\label{eq.repzero}
\frac{\mathbb E[\mathcal N(f_n, [0,2\pi])]}{n} = \mathbb E \left[\mathbb E_X \left[  \mathcal N\left(S_n, \left[0,2\pi\right]\right)\right] \right].
\end{equation}
As a result, the conclusion of Theorem \ref{theo.main.mdep} above is equivalent to the fact that the sequence $\mathbb E \left[\mathbb E_X \left[  \mathcal N\left(S_n, \left[0,2\pi\right]\right)\right] \right]$ converges to $2/\sqrt{3}$ as $n$ goes to infinity. 
\par
\medskip
The first step of the proof, which is the object of the next Section \ref{sec.clt}, consists in showing that, as $n$ goes to infinity, under the product measure $\mathbb P \otimes \mathbb P_X$, the process $(S_n(t))_{t \in [0,2\pi]}$ converges in distribution, in the $\mathcal C^1$ topology towards an explicit smooth limit process. The latter being non-degenerate, this entails that the nodal set associated with $(S_n)$ converges in distribution (as a point process) to the one of the limit. In particular, as $n$ goes to infinity, the sequence of random variables $\mathcal N\left(S_n, \left[0,2\pi\right]\right)$ converge is distribution. 
\par
\medskip
In Section \ref{sec.anticon}, we establish a crucial anti-concentration estimate for $||S_{n}||_{\infty}$, the sup norm of the field. This estimate is derived from Tur\'an's Lemma for trigonometric polynomials associated with combinatorial  considerations. 
\par
\medskip
In Section \ref{sec.ui}, we combine the results of both Sections \ref{sec.clt} and \ref{sec.anticon} to establish that the sequence $(\mathcal N\left(S_n, \left[0,2\pi\right]\right))_{n\geq 1}$ is uniformly integrable under $\mathbb P \otimes \mathbb P_X$, or more precisely that is bounded in $L^{1+\epsilon}(\mathbb P \otimes \mathbb P_X)$ for some small $\epsilon>0$. This will complete the proof of Theorem \ref{theo.main.mdep}.
\par
\medskip
Finally, in the last Section \ref{sec.gauss}, using total variations bounds between Gaussian vectors, we approximate the polynomial associated with the sequence $a_k=H(X_k)$ by its analogue where the covariance function of the Gaussian sequence $(X_k)$ is truncated at a threshold $m_n$. Doing so, we deduce Theorem \ref{theo.main.gauss} from Theorem \ref{theo.main.mdep} with the fast decay of the covariance function $\rho_G$ compensating the small growth of $m_n$ as $n$ goes to infinity.

\section{A Central Limit Theorem}\label{sec.clt}

The main objective of this section is to establish the following functional Central Limit Theorem and its corollary on the convergence of nodal sets. 
\begin{theorem}\label{theo.clt}
Let us consider $(a_{k,n})_{k \geq 1, n \geq 1}$ and $(b_{k,n})_{k \geq 1, n \geq 1}$ two independent arrays of $m_n-$dependent random variables satisfying Conditions (A.1)--(A.3), where $m_n = O(n^{\gamma})$ as $n$ goes to infinity, with $\gamma<\frac{\eta}{2(1+\eta)}$.
Then under the probability measure $\mathbb P \otimes \mathbb P_X$, the random process $(S_n(t))_{t \in [0,2\pi]}$ converges in distribution in the functional space $\mathcal{C}^1([0,2\pi])$ towards the process $(S_{\infty}(t))_{t \in [0,2\pi]}=\sqrt{2\pi\psi_\rho(X)}\times(Z_t)_{t\in\mathbb{R}}$ where $(Z_t)_{t\in\mathbb{R}}$ denotes a real centered, stationary Gaussian process whose covariance function is $\sin_c(x)=\frac{\sin(x)}{x}$ and which is independent of the random variable $X$.
\end{theorem}

\begin{remark}
Informally speaking, Theorem \ref{theo.clt} asserts that the field \newline $(S_n(t))_{t \in [0,2\pi]}$ converges in distribution towards a universal $\sin_c$ Gaussian process, multiplied by a random amplitude, which depends on the random locus $X$, via the spectral density $\psi_{\rho}$. The universality of the limit in Theorem \ref{theo.main.mdep} reflects the universality of this $\sin_c$ process and the positiveness of the spectral density.
\end{remark}

As classically done when establishing a functional convergence, the proof of the last Theorem \ref{theo.clt} articulates in two distinct steps: in the next Section \ref{sec.tight}, we establish the tightness of our process in the space $\mathcal{C}^1([0,2\pi])$ and then, in the following Section \ref{sec.clt.mdep}, we establish the convergence in distribution of the finite dimensional marginals. In the last Section \ref{sec.rate}, we moreover exhibit a small rate of convergence of the one-dimensional marginals which will be needed to establish the universality of the asymptotics of the expected number of zeros.
\par
\medskip
\noindent
Since the above functional convergence is valid in the $\mathcal C^1-$topology, one deduces the convergence in distribution of the nodal sets associated with $(S_n)$. In particular, we have the following corollary.
\begin{corollary}\label{cor.conv.nodal}
Under the hypotheses of Theorem \ref{theo.clt} and under the probability measure $\mathbb P \otimes \mathbb P_X$, the number of zeros $\mathcal N(S_n, [0,2\pi])$ converges in distribution to the number of zeros of the limit process $\mathcal N(S_{\infty}, [0,2\pi])$.
\end{corollary}

\begin{proof}[Proof of Corollary \ref{cor.conv.nodal}] Recall that a real smooth stochastic process $(S_t)_{t \in \mathbb R}$ is said to be non-degenerate if almost surely, $S(t)=0$ implies that $S'(t)\neq 0$. Under our Condition (A.3) on the uniform positiveness of the underlying spectral measure $\psi_{\rho}$, the limit process $(S_{\infty}(t))=\sqrt{2\pi\psi_\rho(X)}\times(Z_t)$ appearing in Theorem \ref{theo.clt} is non-degenerate because the Gaussian process $(Z_t)$ is itself non-degenerate, via an immediate application of Bulinskaya's Lemma. The above corollary is then a consequence of the fact that the number of zeros is a continuous functional with respect to the $\mathcal C^1([0,2\pi])$ in a neighborhood of any non-degenerate function, see for example Section 4 of \cite{MR4491425}.
\end{proof}
\subsection{Tightness estimates}\label{sec.tight}

\medskip

In order to establish the tightness in the \newline $\mathcal{C}^1([0,2\pi])-$topology inherent to Theorem  \ref{theo.clt}, it is sufficient to prove that both $(S_n(t))_{t \in [0,2\pi]}$ and $(S_n'(t))_{t \in [0,2\pi]}$ are tight in $\mathcal{C}^0([0,2\pi])$. This will be obtained via a classical Lamperti criterion, see for example Theorem 1 and Remark 1 of \cite{MR1908331}. Recall that under Condition (A.1), the random variables $(a_{k,n})$ and $(b_{k,n})$ are centered with unit variance. For convenience, we set $R_{k,n}(t):=a_{k,n} \cos(k X+ \frac{kt}{n})+b_{k,n} \sin(k X+ \frac{kt}{n})$ so that 
\[
\mathbb E_X\mathbb E\left[\left|S_n(t)-S_n(s)\right|^2\right]=\mathbb{E}\left[\frac{1}{n}\sum_{k,l=1}^n \mathbb{E}_X\left[\left(R_{k,n}(t)-R_{k,n}(s)\right)\left(R_{l,n}(t)-R_{l,n}(s)\right)\right]\right].
\]
One then notices that for any $t\in[0,2\pi]$, $R_{k,n}(t)$ is a linear combination of $\cos(kX),\sin(kX)$ so that, by orthogonality in $L^2([0,2\pi], \mathbb P_X)$, for $k\neq l$ we get 
\[
\mathbb{E}_X\left[\left(R_{k,n}(t)-R_{k,n}(s)\right)\left(R_{l,n}(t)-R_{l,n}(s)\right)\right]=0.
\] 
Besides, when $k=l$, a direct computation yields 
\[
\mathbb E \left[ \left(R_{k,n}(t)-R_{k,n}(s)\right)^2\right]=2 \left( 1 - \cos\left( \frac{k}{n} (t-s) \right) \right) \leq |t-s|^2,
\]
so that 
\[
\mathbb E_X\mathbb E\left[\left|S_n(t)-S_n(s)\right|^2\right] \leq |t-s|^2.
\]
In the exact same manner, replacing $(a_k,b_k)$ by $(\frac{k b_k}{n},\frac{-k a_k}{n})$ in the above computation, we get
\[
\mathbb E_X\mathbb{E}\left[\left|S_n'(t)-S_n'(s)\right|^2\right]\le  |t-s|^2,
\]
hence the tightness of the laws of the processes $((S_n(t))_{t \in [0,2\pi]})_{n \geq 1}$.

\subsection{Finite dimensional convergence in the $m_n-$dependent case}\label{sec.clt.mdep}
Let us now establish the finite dimensional convergence of the process $(S_n(t))_{t \in [0,2\pi]}$ under $\mathbb P \otimes \mathbb P_X$ and  under the assumption of Theorem \ref{theo.clt}.
In other words, assuming that the random coefficients are $m_n-$dependent, we fix $d\ge 1$ as well as $(t_1,t_2,\cdots,t_d)\in[0,2\pi]^d$ and we want to prove that under $\mathbb P \otimes \mathbb P_X$,
\begin{equation}\label{finite-marginal-claim}\left(S_n(t_1),\cdots,S_n(t_d)\right)\xrightarrow[n\to\infty]{\text{Law}} \sqrt{2\pi\psi_\rho(X)}\left(Z_{t_1},\cdots,Z_{t_d}\right),
\end{equation}
where we recall that $\psi_\rho$ is the limit spectral density appearing in \eqref{eq.lowerb}, $X$ is uniformly distributed on $[0,2\pi]$ and independent of $(Z_t)_{t\ge 0}$ which is a Gaussian process with correlation function $\sin_c(x)=\frac{\sin(x)}{x}$. 
\par
\medskip
Recall that, under Condition (A.3), $\mathbb{P}_X$ almost surely, we have $\psi_\rho(X)\geq \kappa_{\rho}>0$.
To establish the convergence \eqref{finite-marginal-claim}, we will in fact prove that $\mathbb P_X$ almost surely, we have 
\begin{equation}\label{finite-marginal-claim2}\frac{1}{\sqrt{2\pi\psi_\rho(X)}}\left(S_n(t_1),\cdots,S_n(t_d)\right)\xrightarrow[n\to\infty]{\text{Law}} \left(Z_{t_1},\cdots,Z_{t_d}\right),
\end{equation}
the latter convergence \eqref{finite-marginal-claim2} being equivalent, via the use of characteristic functions, to the following convergence, $\mathbb P_X$ almost surely, for all fixed vector $\left(\xi_1,\cdots,\xi_d\right)\in\mathbb{R}^d$ 
\[
\mathbb{E}\left[e^{i \sum_{k=1}^d \frac{\xi_k S_n(t_k)}{\sqrt{2\pi \psi_\rho(X)}}}\right]
\xrightarrow[n \to +\infty]{} \exp\left(-\frac{1}{2} \times \sum_{k,l=1}^d \xi_k \xi_l \sin_c(t_k-t_l)\right).
\]
Using the change of variables $\xi_k \to \sqrt{2\pi \psi_\rho(X)} \xi_k$, we will indeed get that $\mathbb P_X$ almost surely
\[
\mathbb{E}\left[e^{i \sum_{k=1}^d \xi_k S_n(t_k)}\right]
\xrightarrow[n \to +\infty]{} \exp\left(-\frac{1}{2} \times 2\pi \psi_{\rho}(X)\times \sum_{k,l=1}^d \xi_k \xi_l \sin_c(t_k-t_l)\right).
\]
Taking the expectation with respect to $\mathbb P_X$, by dominated convergence, we will then have 
\[
\mathbb E_X \left[ \mathbb{E}\left[e^{i \sum_{k=1}^d \xi_k S_n(t_k)}\right]\right]
\xrightarrow[n \to +\infty]{} \mathbb E_X \left[ \exp\left(-\frac{1}{2} \times 2\pi \psi_{\rho}(X)\times \sum_{k,l=1}^d \xi_k \xi_l \sin_c(t_k-t_l)\right)\right],
\]
which is indeed equivalent to \eqref{finite-marginal-claim}. In order to establish the convergence \eqref{finite-marginal-claim2}, we will use a CLT for sums of $m_n-$dependent random variables which is a widely explored topic, see for example \cite{janson2021central,MR4710537,liu2022wasserstein,liu2023smooth} and the references therein. 
We write the sum of interest under the form
\[
\sum_{k=1}^d  \frac{\xi_k S_n(t_k)}{ \sqrt{2\pi \psi_{\rho}(X)}}=\sum_{k=1}^n X_{k,n},
\]
where
{\small 
$$X_{k,n}:=\frac{1}{\sqrt{2\pi n\psi_{\rho}(X)}}\left(a_{k,n}\left(\sum_{l=1}^d \xi_l \cos\left(k X +\frac{k t_l}{n}\right)\right)+b_{k,n} \left(\sum_{l=1}^d \xi_l \sin\left(k X +\frac{k t_l}{n}\right)\right)\right).$$
}
Since by our assumptions, the sequences $(a_{k,n})_{k \geq 1}$ and ($b_{k,n})_{k \geq 1}$ are here $m_n-$dependent under $\mathbb P$, so is the sequence $(X_{k,n})_{1\le k \le n}$, under $\mathbb P$ and conditionally to $X$. 
Let us show that $\mathbb P_X$ almost surely, under $\mathbb P$, the sequence $\sum_{k=1}^n X_{k,n}$ indeed satisfies a CLT.
In order to apply Theorem 1.4. of \cite{janson2021central}, setting $\sigma_n^2:=\mathbb E\left[(\sum_{k=1}^n X_{k,n})^2\right]$, we only need to check that $\sigma_n^2>0$ for $n$ large enough and then the validity of the so-called \textit{Lindeberg condition} 
\begin{equation}\label{Lindeberg}
\forall \varepsilon>0, \quad \lim_{n\to\infty}\frac{m_n}{\sigma_n^2}\sum_{k=1}^n \mathbb E\left[X_{k,n}^2 \mathds{1}_{\{|X_{k,n}|>\varepsilon \frac{\sigma_n}{m_n}\}}\right]=0.
\end{equation}
Let us first determine the asymptotic behaviour of $\sigma_n$ as $n$ goes to infinity. Such a computation has already been performed in \cite{MR4491425}. Indeed, following Lemma 3.5 of the latter reference, we have the representation
\[
\sigma_n^2 = \frac{ K_n^{t,\xi} \ast \mu_{\rho_n}(X)}{\psi_{\rho}(X)} = \frac{1}{2\pi\psi_{\rho}(X) }\left( \int_0^{2\pi} K_n^{t,\xi}(x)dx\right) \bar{K}_n^{t,\xi}\ast \mu_{\rho_n}(X),
\]
with 
\[
K_n^{t,\xi}(x):=\frac{1}{n}\left| \sum_{l=1}^{d}\xi_l e^{i\frac{(n+1)}{2n}t_l} \frac{\sin\left(\frac{n}{2}(x+\frac{t_l}{n})\right)}{\sin\left(\frac{x+\frac{t_l}{n}}{2}\right)}\right|^2, \quad 
\bar{K}_n^{t,\xi}(x):=\frac{{2\pi} \, K_n^{t,\xi}(x)}{\int_0^{2\pi}K_n^{t,\xi}(x)dx}.
\]
By Lemma 5.1 of the same reference, the function $x \mapsto \bar{K}_n^{t,\xi}(x)$ is a good trigonometric kernel with respect to the normalized Lebesgue measure on $[0, 2\pi]$.  As a result, since under our assumptions we have $\mu_{\rho_n}(dx) = \psi_{\rho_n}(x) dx$ and $\mu_{\rho}(dx) = \psi_{\rho}(x) dx$, using Condition (A.3), we deduce that as $n$ goes to infinity
\[
| \bar{K}_n^{t,\xi} \ast  \mu_{\rho_n}(X) - \bar{K}_n^{t,\xi} \ast  \mu_{\rho_n}(X)| \leq ||\psi_{\rho_n} - \psi_{\rho}||_{\infty} =o(1).
\]
Otherwise, using Lemma 3.6 of \cite{MR4491425}, we have that $\mathbb P_X$ almost surely, as $n$ goes to infinity 
\[
\lim_{n \to +\infty} \bar{K}_n^{t,\xi} \ast \mu_{\rho}(X) = \psi_{\rho}(X).
\]
Finally, by Lemma 3.7 of the same reference again, we have as $n$ goes to infinity
\[
\int_0^{2\pi} K_n^{t,\lambda}(x) dx =2 \pi \sum_{k, l=1}^d \xi_k \xi_l \sin_c\left(t_k-t_l\right) + O\left( \frac{1}{n} \right).
\]
As a result, we can conclude that $\mathbb P_X$ almost surely, as $n$ goes to infinity, we have 
\begin{equation}\label{limit-var}
\sigma_n^2 \xrightarrow[n\to\infty]~\sum_{k,l=1}^d \xi_k \xi_l\sin_c(t_k-t_l).
\end{equation}
In particular, for fixed vectors $(t_k)_{1\leq k\leq d}$ and $(\xi_k)_{1\leq k\leq d}$, we have   $$\liminf_{n \to +\infty} \sigma_n^2>0.$$ 
Let us now check the validity of the Lindeberg condition \eqref{Lindeberg}. 
Recall again that, under Condition (A.3), $\mathbb{P}_X$ almost surely, we have $\psi_\rho(X)\geq \kappa_{\rho}>0$. Therefore, we have the rough upper bound
\[
|X_{k,n}|\le \frac{|a_{k,n}|+|b_{k,n}|}{\sqrt{2\pi \kappa_{\rho} n}}\sum_{l=1}^d |\xi_l|=:\frac{C_{\rho}}{\sqrt{n}}\left(|a_{k,n}|+|b_{k,n}|\right),
\]
so that under Condition (A.1), we have for $\eta>0$ small enough 
\[
\mathbb E[|X_{k,n}|^{2+\eta}] \leq \left( \frac{4C_{\rho}^2}{n} \right)^{1+\frac{\eta}{2}} K_{\eta}.
\]
For $\varepsilon>0$, by H\"older and Markov inequalities we get 
\[
\begin{array}{ll}
\displaystyle{
\mathbb{E}\left[X_{k,n}^2\textbf{1}_{\{|X_{k,n}|>\varepsilon \frac{\sigma_n}{m_n}\}}\right]} & \displaystyle{\le \mathbb E[|X_{k,n}|^{2+\eta}]^{\frac{2}{2+\eta}}\mathbb{P}\left(|X_{k,n}|>\varepsilon \frac{\sigma_n}{m_n}\right)^{\frac{\eta}{2+\eta}}}\\
\\
& \displaystyle{ \leq \mathbb E[|X_{k,n}|^{2+\eta}] \left( \frac{m_n}{\sigma_n \varepsilon}\right)^{\eta}}.
\end{array}
\]
Summing over $k$, we thus get 
\begin{equation}\label{eq.bound.lindeberg}
\frac{m_n}{\sigma_n^2}\sum_{k=1}^n \mathbb E\left[X_{k,n}^2 \mathds{1}_{\{|X_{k,n}|>\varepsilon \frac{\sigma_n}{m_n}\}}\right] \leq 
\left( 4C_{\rho}^2\right)^{1+\frac{\eta}{2}} K_{\eta} \frac{1}{\sigma_n^{2+\eta} \varepsilon^{\eta}}\frac{m_n^{1+\eta}}{n^{\eta/2}}.
\end{equation}
Since $\liminf_{n \to +\infty} \sigma_n^2>0$ and $\varepsilon$ is fixed here, if $m_n = O(n^{\gamma})$ with $\gamma<\eta/2(1+\eta)$ the right hand side of the last equation goes to zero as $n$ goes to infinity and Lindeberg condition \eqref{Lindeberg} indeed applies. Therefore, Theorem 1.4 of \cite{janson2021central} ensures that $\mathbb{P}_X$ almost surely, the convergence \eqref{finite-marginal-claim2} holds, in other words
\begin{equation*}\label{CLT-m-dependent}
\sum_{k=1}^n X_{k,n}\xrightarrow[n\to\infty]{\text{Law under}~\mathbb{P}}~\mathcal{N}\left(0,\sum_{k,l=1}^d \xi_k\xi_l\sin_c(t_k-t_l)\right).
\end{equation*}
Associated with the tightness estimates established in the last Section \ref{sec.tight}, this concludes the proof of Theorem \ref{theo.clt}.

{
\begin{remark}\label{rem.allmoments}
Note that in the case where the random variables $(a_{k,n})$ and $(b_{k,n})$ have uniformly bounded moments of all orders, which happens for example in the Rademacher case, then the parameter $\eta$ in Condition (A.1) can be chosen arbitrarily large, so that the conclusion of the above  Theorem \ref{theo.clt} is valid for a threshold $m_n = O(n^{\gamma})$ with $\gamma$ arbitrarily close to $1/2$.
\end{remark}
}
\subsection{Quantification of the convergence} \label{sec.rate}

We now establish a polynomial rate of convergence for the one dimensional marginals in the above functional Central Limit Theorem.
This rate will be exploited in the next Section \ref{sec.using.rate.clt} to establish the uniform integrability of the sequence $(\mathcal N(g_n, [0,2\pi]))_{n \geq 1}$ and hence the universality of the asymptotics of the expected number of zeros of $f_n$.
We denote by $\text{d}_K^{\mathbb P}$ the Kolmogorov distance under $\mathbb P$ 
\[
\mathrm{d}_K^{\mathbb P} \left(U, V\right)=\sup_{s \in \mathbb R} \left| \mathbb P  \left(U>s  \right)- \mathbb P  \left( V>s  \right)\right|.
\]
{
\begin{prop}\label{pro.rate}
Let us assume that the sequence $(m_n)$ satisfies $m_n=O(n^{\gamma})$  with $\gamma<\frac{\eta}{2(1+\eta)}$.  Let $N$ be a real random variable with standard Gaussian distribution  under $\mathbb P$ and independent of $X$. Then, $\mathbb P_X$ almost surely, as $n$ goes to infinity, we have 
\[
\mathrm{d}_K^{\mathbb P} \left(\frac{S_n(0)}{\sqrt{2\pi \psi_{\rho}(X)}}, N\right)  = O\left( \frac{1}{n^{\left( \frac{1+\eta}{4+6\eta}\right) \left( \frac{\eta}{2(1+\eta)}-\gamma \right)}} \right).
\]
\end{prop}
}
\begin{proof}Let us adopt here the same notations as in the last Section \ref{sec.clt.mdep}, with $d=1$, $t_1=0$ and $\xi_1=1$ and set 
\[
\frac{S_n(0)}{ \sqrt{2\pi \psi_{\rho}(X)}}=\sum_{k=1}^n X_{k,n}
\]
where
\[
X_{k,n}:=\frac{1}{\sqrt{2\pi n\psi_{\rho}(X)}}\left(a_{k,n} \cos\left(k X\right)+b_{k,n}  \sin\left(k X\right)\right).
\]
Combining Lemma 2 and Theorem 4 of \cite{MR4710537}, we have the following upper bound, valid for all $\varepsilon_n>0$
\begin{equation}\label{eq.upper.U}
\mathrm{d}_K^{\mathbb P} \left(\frac{S_n(0)}{\sqrt{2\pi \psi_{\rho}(X)}}, N\right)^2 \leq 120 \left
(\varepsilon_n^{1/3} + 12 U_n(\varepsilon_n/2)^{1/2}\right),  
\end{equation}
where $U_n$ is the term controlling the Lindeberg condition, namely
\[
U_n(\varepsilon_n):=\frac{m_n}{\sigma_n^2}\sum_{k=1}^n \mathbb E\left[X_{k,n}^2 \mathds{1}_{\{|X_{k,n}|>\varepsilon_n \frac{\sigma_n}{m_n}\}}\right]\quad \text{and} \quad \sigma_n^2 =\mathbb E\left[ \left(\sum_{k=1}^n X_{k,n}\right)^2\right].
\]
By Equation \eqref{limit-var}, as $n$ goes to infinity, we have here $\lim_{n \to +\infty} \sigma_n^2=1$.
{By Equation \eqref{eq.bound.lindeberg}, recalling that $\varepsilon=\varepsilon_n$ now,  we have  
\[
U_n(\varepsilon_n) = O \left(  \frac{1}{ \varepsilon_n^{\eta}}\frac{m_n^{1+\eta}}{n^{\eta/2}}\right).
\]
For $\varepsilon_n =n^{-\delta}$ and $m_n=n^{\gamma}$, this last term goes to zero as soon as $\delta \eta + \gamma (1+\eta)-\eta/2<0$. Balancing the two factors on the right hand side of Equation \eqref{eq.upper.U}, we get finally  
\[
\mathrm{d}_K^{\mathbb P}\left(\frac{S_n(0)}{\sqrt{2\pi \psi_{\rho}(X)}}, N\right)=O(\varepsilon_n^{1/6}) = O\left(n^{-\left( \frac{1+\eta}{4+6\eta}\right) \left( \frac{\eta}{2(1+\eta)}-\gamma \right)} \right).
\]
}
\end{proof}

{
\begin{cor}\label{cor.rate}Under the hypotheses of Theorem \ref{theo.clt}, for any $0<\gamma<\frac{\eta}{2(1+\eta)}$, uniformly in $s>0$ and for $n$ large enough, we have
\[
 \mathbb P \otimes \mathbb P_X \left(|S_n(0)|<s  \right) \leq \frac{2s }{\pi\sqrt{\kappa_{\rho}} }\wedge 1 +O\left(n^{-\left( \frac{1+\eta}{4+6\eta}\right) \left( \frac{\eta}{2(1+\eta)}-\gamma \right)} \right).
\]
\end{cor}
\begin{proof}Following the last Proposition \ref{pro.rate}, $\mathbb P_X$ almost surely and uniformly in $s>0$, we have 
\[
\left| \mathbb P \left(\frac{|S_n(0)|}{\sqrt{2\pi \psi_{\rho}(X)}}<s\right) - \mathbb P \left( |N|<s \right)\right|=O\left(n^{-\left( \frac{1+\eta}{4+6\eta}\right) \left( \frac{\eta}{2(1+\eta)}-\gamma \right)} \right).
\]
Using the change of variable $s \mapsto s \times \sqrt{2\pi \psi_{\rho}(X)}$, $\mathbb P_X$ almost surely and uniformly in $s>0$, we have similarly 
\[
\left| \mathbb P \left(|S_n(0)|<s\right) - \mathbb P \left( \sqrt{2\pi \psi_{\rho}(X)} |N|<s \right)\right|=O\left(n^{-\left( \frac{1+\eta}{4+6\eta}\right) \left( \frac{\eta}{2(1+\eta)}-\gamma \right)} \right).
\]
Now since the density of the standard Gaussian variable $N$ is bounded by $1/\sqrt{2\pi}$, recalling that under Condition (A.3), $\mathbb P_X$ almost surely we have $\psi_{\rho}(X) \geq \kappa_{\rho}>0$, we get
\[
\mathbb P \left( \sqrt{2\pi \psi_{\rho}(X)} |N|<s \right) \leq \frac{2s }{\pi\sqrt{\kappa_{\rho}}} \wedge 1,
\]
hence the result.
\end{proof}
}

{
\begin{remark}
As already noticed in Remark \ref{rem.allmoments} above, in the case where the random variables $(a_{k,n})$ and $(b_{k,n})$ have uniformly bounded moments of all orders, the parameter $\eta$ in Condition (A.1) can be chosen arbitrarily large, so that in both Proposition \ref{pro.rate} and Corollary \ref{cor.rate}, the estimates are in fact of order $n^{-(1/2-\gamma)/6}$.
\end{remark}
}
\section{Anti-concentration estimates} \label{sec.anticon}
The goal of this section is to establish strong anti-concentration estimates for  the sup norm $||S_n ||_{\infty}$ of our random field, under the hypothesis that the coefficients $(a_{k,n})_{k}$ and $(b_{k,n})_{k}$ form $m_n-$ dependent sequences.
As mentioned in the introduction, the main difficulty lies in the fact that we are working with generally distributed and dependent random variables, which may be discrete. Consequently, both classical tools such as the Hal\'asz method and characteristic function estimates and recently developed tools such as the inverse Littlewood-Offord machinery \cite{nguyen2013small} are not applicable. Our method is of combinatorial and analytic nature, based on Tur\'an's Lemma, the $m_n$-dependency hypothesis of the coefficients being used to lower the degree of the involved trigonometric polynomials.

 \subsection{Random coefficients with values in a finite set}
To give a transparent view of our overall strategy here, we first establish the desired anti-concentration estimates in the simpler case where the random coefficients $(a_{k,n})$ and $(b_{k,n})$ take values in a finite set $\mathcal A\subset \mathbb Z$ of cardinal $A$, which encompasses for example the Rademacher case where the variables take values in the set $\{-1,+1\}$. The general case without restriction on the support of the variables will be treated in the next Section \ref{sec.general.anti}.

 \begin{thm} \label{theo.anti}
Let us consider $(a_{k,n})_{k \geq 1, n \geq 1}$ and $(b_{k,n})_{k \geq 1, n \geq 1}$ be two independent arrays of $m_n-$dependent random variables. We assume that all these variables take values in a finite set $\mathcal A\subset \mathbb Z$ of cardinality $A$ with $\mathbb P(a_{k,n} = a)\le c<1$ for some constant $c$, for all $k,n$ and all $a\in \mathcal A$. Let $\delta=\delta_n>0$ and $N=N_n\le n$ be parameters satisfying 
 	\begin{equation}\label{cond:1}
 		\left(\sqrt{n} \, \delta_n \right)^{1/(2(m_n+1)N_n)} A^{2(m_n+1)N_n} =o(1/n).
 	\end{equation}
Then, for $n$ large enough and for all realizations $X(x)$ of the variable $X$, there exists a point $t=t(x) \in [0, 2\pi]$ such that
\[
\sup_{z\in \mathbb R} \mathbb P(|S_n(t)-z|\le \delta_n)\le c^{N_n}A^{m_n}.
\]
 \end{thm}
 
Let us illustrate the pertinence of the last result by applying it to random Littlewood polynomials, i.e. random trigonometric polynomials with symmetric Rademacher coefficients. In this case, we have $\mathcal A=\{\-1,+1\}$, $A=2$, $c=1/2$ and with the parameters $z=0$, $\delta_n=\frac{1}{\lfloor n^{\beta} \rfloor !}$, $N_n=n^{\alpha}$, $m_n=n^{\gamma}$ with $0<2(\alpha+\gamma)<\beta<1$, for which Condition \eqref{cond:1} is satisfied, we deduce the following small-ball estimate which bounds the probability that $||S_n ||_{\infty}$ is small.
 \begin{corollary}\label{cor.anti}
Let us consider $(a_{k,n})_{k \geq 1, n \geq 1}$ and $(b_{k,n})_{k \geq 1, n \geq 1}$ be two independent arrays of $m_n-$dependent symmetric Rademacher random variables, where $m_n = O(n^{\gamma})$ as $n$ goes to infinity for $0\leq \gamma <1/4$. Then, for $n$ large enough and for all realizations $X(x)$ of the variable $X$, there exists $t=t(x) \in [0, 2\pi]$ such that for all $0<2(\alpha+\gamma)<\beta<1$
\[
 \mathbb P\left(|S_n(t)|\le \frac{1}{\lfloor n^{\beta} \rfloor !} \right) \leq 2^{n^{\gamma}-n^{\alpha}}.
\]
In particular, for $n$ large enough and for $\alpha>\gamma$ and $0<2(\alpha+\gamma)<\beta<1$, we have
\[
\mathbb P \otimes \mathbb P_X\left(||S_n ||_{\infty} \le \frac{1}{\lfloor n^{\beta} \rfloor !} \right) \leq 2^{-n^{\alpha}(1+o(1))}.
\]
 \end{corollary}
\if{
\begin{remark}\label{rem.gamma}
In order to be pertinent, Corollary \ref{cor.anti} naturally requires the exponents to satisfy $\alpha>\gamma$ so that the condition $0<2(\alpha+\gamma)<\beta<1$ indeed implies that $\gamma<1/4$.
\end{remark}
}\fi

\begin{proof}[Proof of Theorem \ref{theo.anti}]
Since our goal is to show that, for all realizations of $X$, there exists $t$ such that
\[
\sup_{z\in \mathbb R} \mathbb P(|S_n(t)-z|\le \delta_{n})\le c^{N_{n}}A^{m_{n}},
\]
recalling the independence of the sequences $(a_{k,n})$ and $(b_{k,n})$ under $\mathbb P$ and considering the change of variable  $z \mapsto z-\sum_{k=1}^n b_{k,n} \sin(kX +kt/n)$, it is sufficient to show for all realizations of $X$, there exists $t$ such that
\[
\sup_{z\in \mathbb R} \mathbb P\left(\left|\widetilde{S}_n(t)-z\right|\le \delta \right)\le c^{N_{n}}A^{m_{n}},
\]
where we have set 
\[
\widetilde{S}_n(t):=\frac{1}{\sqrt{n}} \sum_{k=1}^n a_{k,n} \cos\left( k X +\frac{kt}{n} \right).
\]
We decompose $\widetilde{S}_n(t)$ into the sum $\widetilde{S}_n(t)=P_n(t)+Q_n(t)+R_n(t)$ with 
\[
\begin{array}{ll}
\displaystyle{P_n(t):=\frac{1}{\sqrt{n}} \sum_{k=1}^{(m_{n}+1)N_{n}} a_{k,n}  \cos\left( k X +\frac{kt}{n} \right)}, 
\\
\\
\displaystyle{ Q_n(t):= \frac{1}{\sqrt{n}} \sum_{k=(m_{n}+1)N_{n}+1}^{(m_{n}+1)N_{n}+m_{n}} a_{k,n}  \cos\left( k X +\frac{kt}{n} \right)},\\
\\
\displaystyle{R_n(t):= \frac{1}{\sqrt{n}} \sum_{k=(m_{n}+1)N_{n}+m_{n}+1}^n a_{k,n}  \cos\left( k X +\frac{kt}{n} \right)}.
\end{array}
\]
Thanks to the $m_{n}-$dependent hypothesis, observe that the random variables $P_n(t)$ and $R_n(t)$ are independent under $\mathbb P$. Therefore, exploring the random coefficients of $Q_n(t)$, we have 
\begin{eqnarray}
\mathbb P(|\widetilde{S}_n(t)-z|\le \delta_{n}) &= & \sum_{\sigma\in \mathcal A^{m_{n}}}  \mathbb P(|\widetilde{S}_n(t)-z|\le \delta_{n}, (a_{(m_{n}+1)N_{n}+1,n}, \dots, a_{(m_{n}+1)N_{n}+m_{n},n}) = \sigma)\notag\\
&\le&A^{m_{n}} \sup_{z\in \mathbb R} \mathbb P(|P_n(t)+R_n(t)-z|\le \delta_{n})\notag\\
&\le&A^{m_{n}} \sup_{z\in \mathbb R} \mathbb P(|P_n(t)-z|\le \delta_{n})\notag,
\end{eqnarray}
where 
\begin{itemize}
\item in the first line we used the total probability formula,
\item in the second line we used the (deterministic) change of variables $z\mapsto z-Q_n(t)$, 
\item in the the last line, we used the change of variable $z \mapsto z-R_n(t)$ and the independence of $P_n(t)$ and $R_n(t)$ under $\mathbb P$. 
\end{itemize}
The fact the variables $(a_{k,n})$ take values in a finite set here thus allows us to drastically lower the degree of the trigonometric polynomials involved, from $O(n)$ to $O(m_n N_n)$.
 \par
\medskip
Now, let us fix $z\in \mathbb R$ and consider the set of all possible realizations of the independent random variables $(a_{(m+1)j,n})_{1\leq j \leq N}$ that can be completed to get $|P_n(t)-z|\le \delta_{n}$. To this end and in order to clearly identify what is random from what is not in the sequel, let us introduce the following notation, for any deterministic sequence $\sigma \in \mathbb R^{(m_{n}+1)N_{n}}$, we set 
\[
p_{\sigma}(t):=\frac{1}{\sqrt{n}} \sum_{k=1}^{(m_{n}+1)N_{n}} \sigma_k \cos\left(k t\right). 
\]
In particular, given the sequence $(a_{k,n})_k$, setting $a=(a_{1,n}, \ldots, a_{(m+1)N,n}) \in \mathcal A^{(m_{n}+1)N_{n}}$ and comparing with the previously introduced notations, we have the identification
\[
P_n(t)=p_a \left( X+\frac{t}{n}\right).
\]
For a given $t \in [0, 2\pi]$, we will say that $(\alpha_1, \dots, \alpha_{N_{n}})\in \mathcal A^{N_{n}}$ is a $t-$feasible tuple if there exists $\sigma\in \mathcal A^{(m_{n}+1)N_{n}}$ such that 
\begin{equation}\label{eq.feasible}
\sigma_{(m_{n}+1)j}=\alpha_j \;\; \text{for}\;\;  1\leq j \leq N_{n}, \quad \text{and} \;\; |p_\sigma(t)-z|\le \delta_{n}.
\end{equation}
 	
The following deterministic lemma shows that there exist ``good points $t$'' where the number of feasible tuples can be drastically upper bounded.

\begin{lemma}\label{lem.exists}In any interval subinterval of width $2\pi/n$ of $[0, 2\pi]$, there exists a point $t$ for which there exists at most one $t-$feasible tuple $(\alpha_1, \dots, \alpha_N)$.
\end{lemma}

\begin{proof}[Proof of Lemma \ref{lem.exists}] The proof relies on the celebrated Tur\'an's Lemma for trigonometric polynomials, see for example \cite[Chapter I]{MR1246419}. Recall that if 
\[
q(t) := \sum_{k=0}^{d} b_k e^{i\lambda_k t}, \quad b_k \in \mathbb C, \quad\lambda_0<\lambda_1<\dots< \lambda_d \in \mathbb R,
\]
then Tur\'an's Lemma asserts that there exists an absolute constant $C$ such that for any interval $J\subset \mathbb R$ and any measurable subset $E\subset J$ of positive measure, we have
\begin{equation}\label{eq.turan}
\sup_{t\in J} |q(t) |\le \left (\frac{C|J|}{|E|}\right )^{d}\sup _{t\in E} |q(t) |,
\end{equation}
where $|E|$ denotes here the Lebesgue measure of the set $E$.
Let us now introduce the difference set
\[
D_{\mathcal A}:= \{a-a', \; (a,a')\in \mathcal A \times \mathcal A\},  \;\; D_{\mathcal A}^*:=D_{\mathcal A} \backslash \{0\}, 
\]
and set $\eta_{\mathcal A}:=\min\{ |s|, \; s \in D_{\mathcal A}^*\}\ge 1$ the minimum non-zero modulus. To any non-zero $\tau\in D_{\mathcal A}^{(m_{n}+1)N_{n}}$, we associate the (deterministic) trigonometric polynomial $p_{\tau}$ and the set $E_{\tau}:=\{t\in [0, 2\pi]: |p_{\tau}(t)|\le 3\delta_{n}\}$.
Since we have
\[
\sup_{t \in [0,2\pi]} |p_{\tau}(t)|^2 \geq \frac{1}{2\pi} \int_{0}^{2\pi} p_{\tau}^{2}(t)dt = \frac{1}{2n}\sum_{k=1}^{(m_{n}+1)N_{n}} \tau_k^2 \geq \frac{\eta_{\mathcal A}^2}{2n},
\]
applying Tur\'an's Lemma with $q=p_{\tau}$, $J=[0,2\pi]$ and $E=E_{\tau}$, we obtain that there exists a universal constant $C$ such that
\[
\frac{\eta_{\mathcal A}}{\sqrt{2n}} \leq  \sup_{t \in [0,2\pi]} |p_{\tau}(t)|  \leq \left (\frac{2\pi C}{|E_{\tau}|}\right )^{2(m_{n}+1)N_{n}}\sup _{t\in E_{\tau}} |p_{\tau}(t) |\le \left (\frac{2\pi C}{|E_{\tau}|}\right )^{2(m_{n}+1)N_{n}} 3\delta_{n}.
\]
Therefore, uniformly in $\tau$, we have
\[
|E_{\tau}|\le 2\pi C \left( \frac{3\sqrt{2n} \delta_{n}}{\eta_{\mathcal A}}\right)^{1/(2(m_{n}+1)N_{n})}.
\]
Taking the union over all possible choices of coefficients $\tau$, we obtain using Condition \eqref{cond:1} that
\[
\sum_{\tau} |E_{\tau}|=O\left(  \left(\sqrt{n}\delta_{n}\right)^{1/(2(m_{n}+1)N_{n})} A^{2(m_{n}+1)N_{n}}\right)=o(1/n).
\]
As a result, in any interval of width $2\pi/n$, for $n$ large enough, there exists a point $t$ that falls out of all of the $E_{\tau}$, that is which is such that for all non zero $\tau \in D_{\mathcal A}^{(m+1)N}$, we have 
 \begin{equation}\label{eq.contra}
|p_{\tau}(t) | > 3\delta_{n}.
 \end{equation} 
For this particular $t$, we will show that there is indeed at most one $t-$feasible tuple. Assume by contradiction that there exist two distinct $t-$feasible tuples $(\alpha_i)_{1\leq i \leq N}$ and $(\alpha'_i)_{1\leq i \leq N}$. By definition, they correspond to two distinct vectors $\sigma$ and $\sigma'$ in $\mathcal A^{(m+1)N}$ for which $\alpha_j = \sigma_j$, $\alpha_j' = \sigma_j'$ for all $1 \leq j\leq (m+1)N$ and the associated polynomials $p_{\sigma}$ and $p_{\sigma'}$ satisfy
\[
|p_{\sigma}(t)-z|\le \delta_{n},\quad \ |p_{\sigma'}-z|\le \delta_{n}.
\] 
Taking difference, noting that $\sigma-\sigma' \in D_{\mathcal A}^{(m+1)N}$ is non-zero, we obtain
\[
|p_{\sigma-\sigma'}(t)|=|p_{\sigma}(t)-p_{\sigma'}(t)|\le 2\delta_{n},
\]
which contradicts \eqref{eq.contra}, hence the result.
\end{proof}
 	
Let us go back to the proof of Theorem \ref{theo.anti}. By Lemma \ref{lem.exists}, for any realization $X(x)$ of the variable $X$, there exist a point $t=t(x)$  in the interval $[X(x), X(x)+2\pi/n]$, to which we can associate at most one $t-$feasible tuple (if any), say $(\alpha^{x}_1, \dots, \alpha^{x}_N)$. Note that this association is deterministic in the sense that it does not depend on the space $(\Omega, \mathcal F, \mathbb P)$ thus on the variables $(a_{k,n},b_{k,n})_{k \geq 1}$. Therefore, for any realization $X(x)$ of the variable $X$, there exists $t=t(x)$ such that
\begin{eqnarray}
\mathbb P\left (\left |P_n(t)-z\right |\le \delta_{n}\right )  & = & \sum_{\alpha \in \mathcal A^N}  \mathbb P\left (\left |P_n(t)-z\right |\le \delta_{n}, \; a_{(m_{n}+1)j,n}=\alpha_j, \, 1 \leq j \leq N\right ) \notag\\
 		& \leq &  \mathbb P\left (\left |P_n(t)-z\right |\le \delta_{n}, \;  a_{(m_{n}+1)j,n}=\alpha_j^x, \, 1 \leq j \leq N_{n}\right )  \notag\\
 		&\le &	 \mathbb P\left ( a_{(m_{n}+1)j,n}=\alpha_j^x,\, 1\leq j \leq N\right )\notag\\
 		&\le &c^{N}\text{ by independence}.\notag
\end{eqnarray}
This completes the proof. 
\end{proof}
\subsection{General case} \label{sec.general.anti}
In this section, we prove the corresponding version of Theorem \ref{theo.anti}, without any restriction of the support of the involved random variables.   Let us first state and prove an elementary auxiliary lemma. 

\begin{lm}\label{lem.kappa}
Suppose that $(a_{k,n})$ and $(b_{k,n})$ satisfy Condition (A.1), then there exist constants $\delta_0>0$ and $0\leq \kappa<1$ such that
\[
\sup_{k\geq 1, n \geq 1} \mathbb P(|a_{k,n}-b_{k,n}|<\delta_0) \leq \kappa.
\]
\end{lm}
\begin{proof}[Proof of Lemma \ref{lem.kappa}]Under Condition (A.1), decomposing the expectation and using H\"older inequality, we have for all $\delta>0$ 
\[
\begin{array}{ll}
2 & =\mathbb E[ |a_{k,n}-b_{k,n}|^2]  = \mathbb E[ |a_{k,n}-b_{k,n}|^2\mathds{1}_{|a_{k,n}-b_{k,n}|\geq \delta}]+ \mathbb E[ |a_{k,n}-b_{k,n}|^2\mathds{1}_{|a_{k,n}-b_{k,n}|< \delta}]\\
\\
& \leq \mathbb E[ |a_{k,n}-b_{k,n}|^{2+\eta}]^{\frac{2}{2+\eta}} \mathbb P\left(|a_{k,n}-b_{k,n}|\geq \delta\right)^{\frac{\eta}{2+\eta}} + \delta^2 \, \mathbb P\left( |a_{k,n}-b_{k,n}|< \delta\right)\\
\\
& \leq K_{\eta}^{\frac{2}{2+\eta}}\mathbb P\left(|a_{k,n}-b_{k,n}|\geq \delta\right)^{\frac{\eta}{2+\eta}}+\delta^2 \, \mathbb P\left( |a_{k,n}-b_{k,n}|< \delta\right).
\end{array}
\]
Letting $\delta$ go to zero, one gets that $\mathbb P\left(|a_{k,n}-b_{k,n}|\geq \delta\right)$ is lower bounded by a positive constant, uniformly in $k$ and $n$, hence the result.
\end{proof}
We are now ready to state our anti-concentration bound for the general case. Let $\delta_0$ and $\kappa$ be as in Lemma \ref{lem.kappa}.
\begin{thm}\label{thm:anti}Let us consider $(a_{k,n})_{k \geq 1, n \geq 1}$ and $(b_{k,n})_{k \geq 1, n \geq 1}$ two independent arrays of $m_n-$dependent random variables, satisfying Conditions (A.1)--(A.3). Let $N_n\le n-m_n$  and $\delta_0^{2} \geq \delta_n>0$ be deterministic numbers, such that $\delta_n\to 0$ as $n$ goes to infinity.
For $n$ large enough and for all realizations $X(x)$ of the variable $X$, there exists a point $t=t(x) \in [0, 2\pi]$ and a universal constant $C$ such that 
\[
\mathbb P(|S_n(t)|\le \delta_n)\le \left( C  (n \delta_{n})^{\frac{1}{4(N_{n}+m_{n})}}+ \kappa^{\lfloor\frac{N_{n}}{m_{n}+1}\rfloor}\right)^{1/2}.
\]
\end{thm}

\begin{proof} [Proof of Theorem \ref{thm:anti}] 
The overall strategy of the proof is similar to the one of Theorem \ref{theo.anti} above: we will first lower the degree of the involved trigonometric polynomials and then use Tur\'an's lemma to establish small-ball estimates. Note that, in view of exhibiting a point $t \in [0,2\pi]$ satisfying the bound in Theorem \ref{thm:anti}, it is sufficient to show that, if $U$ is  a uniform random variable in $[0,2\pi]$ independent of all the other variables, we have 
\[
\mathbb E_U \mathbb P\left(\left|S_n(U)\right|\le \delta_n \right) =\mathbb P_U \otimes \mathbb P\left(\left|S_n(U)\right|\le \delta_n \right)\le \left( C  (n \delta)^{\frac{1}{4(N+m)}}+ \kappa^{\lfloor\frac{N}{m+1}\rfloor}\right)^{1/2}.
\]
{\it Reducing the degree}. We replace the argument using the finiteness of the support of random variables in Theorem \ref{theo.anti} by using conditional expectations here. 
For simplicity of notations, we drop the subscript $n$ in $\delta_{n}, N_{n}$ and $m_n$.
For $1 \leq p\leq q\leq n$, let us introduce the generated sigma fields $\mathcal F_{p}^q:=\sigma(a_{k,n}, p\leq k \leq q)$ and set $\mathcal G:=\mathcal F_{N+m+1}^n \vee \sigma(b_{k,n}, 1\leq k \leq n)$. Under our $m_n$-dependency condition, note that $\mathcal F_{1}^N$ and $\mathcal G$ are independent. 
If  $\mathbb E \left[ \, \cdot \, | \mathcal G\right]$ classically denotes the conditional expectation, one can then write
\[
\begin{array}{ll}
\displaystyle{\mathbb P_U \otimes \mathbb P\left(\left|S_n(U)\right|\le \delta \right) } & \displaystyle{= \mathbb E_U \left[ \mathbb E \left[ \mathds{1}_{|S_n(U)|\le \delta}\right]\right] = \mathbb E_U \left[ \mathbb E \left[ \mathbb E\left[ \mathds{1}_{|S_n(U)|\le \delta} \big| \mathcal G \right]\right]\right]}\\
\\
& \displaystyle{= \mathbb E \left[ \mathbb E_U \left[ \mathbb E\left[ \mathds{1}_{|S_n(U)|\le\delta}\big| \mathcal G\right]\right]\right].} 
\end{array}
\]
Applying Cauchy--Schwarz inequality under the exterior integrals $\mathbb E$ and $\mathbb E_U$, we get 
\begin{equation}\label{eq.pu}
\mathbb P_U \otimes \mathbb P\left(\left|S_n(U)\right|\le \delta \right) 
\leq  \mathbb E \left[ \mathbb E_U \left[ \mathbb E\left[ \mathds{1}_{|S_n(U)|\le \delta} \big| \mathcal G\right]^2 \right]\right]^{1/2}.
\end{equation}
Conditionally on $\mathcal G$ and $(U,X)$, let us consider $(a_{k,n}')_{1 \leq k \leq N+m}$ an independent copy of the 
vector $(a_{k,n})_{1\leq k \leq N+m}$ (both distributed according to the law of $(a_{k, n})_{1\leq k \leq N+m}$ conditioned on $\mathcal G$ and $(U,X)$). The squared inner integral on the right hand side can then be written as 
\[
\mathbb E\left[ \mathds{1}_{|S_n(U)|\le \delta}\big| \mathcal G\right]^2 = \mathbb E\left[ \mathds{1}_{|S_n(U)|\le \delta} \mathds{1}_{|S_n'(U)|\leq \delta} \big| \mathcal G\right], 
\]
where 
\[
S_n'(U)=S_n(U)+ \underbrace{\frac{1}{\sqrt{n}} \sum_{k=1}^{N+m} \left( a_{k,n}'-a_{k,n}\right) \cos\left( k X +\frac{kU}{n} \right)}_{:=T_n(U)}.
\]
As a result, we have
\[
\mathbb E\left[ \mathds{1}_{|S_n(U)|\le \delta}\big| \mathcal G\right]^2 \leq  \mathbb E\left[ \mathds{1}_{|T_n(U)|\le 2 \delta}  \big| \mathcal G\right],
\]
and injecting this estimate in \eqref{eq.pu}, we get 
\begin{equation}\label{eq.pu2}
\begin{array}{ll}
\displaystyle{
\mathbb P_U \otimes \mathbb P\left(\left|S_n(U)\right|\le \delta \right) }& 
\leq  \displaystyle{\left(\mathbb E \left[ \mathbb E_U \left[ \mathbb E\left[ \mathds{1}_{|T_n(U)|\le 2 \delta}  \big| \mathcal G\right] \right]\right]\right)^{1/2}}\\
\\
& \displaystyle{=\mathbb P_U \otimes \mathbb P \left( |T_n(U)|\le 2 \delta \right)^{1/2}},
\end{array} 
\end{equation}
i.e. the probability of interest can indeed be upper bounded in term of a polynomial with lower degree $O(N+m)$ instead of $O(n)$. \par
\medskip
\noindent
{\it Bad configurations}. Let us now consider the ``good'' event $\mathcal A$ on which
\[
\sum_{k=1}^{N+m} |a_{k,n}-a_{k,n}'|^{2}\ge \delta.
\]
To bound the probability of its complement, note that the above sum is less than $\delta$ implies each individual $|a_{k,n}-a_{k,n}'|$ is less than $\delta^{1/2}$ for all $1\leq k \le N+m$, in particular the terms with $k= (m+1)\ell$ where $\ell \le \lfloor\frac{N}{m+1}\rfloor$. Note that, by $m-$dependency, these terms are mutually independent and moreover they belong to $\sigma(a_{k,n},a_{k,n}', 1 \leq k \leq N)$ and are thus independent of $\mathcal G$. So, their conditional distribution knowing $\mathcal G$ is the same as their original distribution. Therefore, applying Lemma \ref{lem.kappa}, for $\delta^{1/2}\leq \delta_0$, we have 
\[
 \mathbb E\left[ \mathds{1}_{\mathcal A^{c}} \big| \mathcal G \right] \le \mathbb P\left (|a_{(m+1)\ell,n}-a_{(m+1)\ell,n}'|<\delta^{1/2}, 1\leq \ell \leq  \left\lfloor\frac{N}{m+1}\right\rfloor\right ) \le \kappa^{\lfloor\frac{N}{m+1}\rfloor}.
\]
In particular, in view of the middle term in \eqref{eq.pu2}, we get 
\begin{equation}\label{eq:Ac}
\mathbb E \left[ \mathbb E_U \left[ \mathbb E\left[ \mathds{1}_{|T_n(U)|\le 2 \delta} \mathds{1}_{\mathcal A^c} \big| \mathcal G\right] \right]\right]\leq \kappa^{\lfloor\frac{N}{m+1}\rfloor}.
\end{equation}
\par
\medskip
\noindent
{\it Good configurations and Tur\'an's lemma}.	Let us set 
\[
 q_n(\theta):=\frac{1}{\sqrt{n}} \sum_{k=1}^{N+m} \left( a_{k,n}'-a_{k,n}\right) \cos\left( k \theta \right), \;\; \text{so that} \;\; T_n(U) = q_n(X+U/n).
\]
On the good event $\mathcal A$, we have
	$$\frac{1}{2\pi} \int_{0}^{2\pi} q_n(\theta)^2 d\theta =\frac{1}{n} \sum_{k=1}^{N+m} |a_{k,n}-a_{k,n}'|^{2}\ge \frac{\delta}{n}.$$
By Tur\'an's Lemma, with $q=q_n$, $J=[0,2\pi]$ and $E=\{t \in [X,X+2\pi/n], \, |q_n(t)| \leq 2\delta\}$, we get 
\[
\sqrt{\frac{\delta}{n}} \leq \sup_{t \in [0,2\pi]} |q_n(t)| \leq \left(  \frac{2\pi C}{|E|} \right)^{2(m+N)}\sup_{t \in E} |q_n(t)| \leq \left(  \frac{2\pi C}{|E|} \right)^{2(m+N)}\times 2\delta
\]
so that 
\[
|E| \leq 2\pi C \times 2^{\frac{1}{2(N+m)}}(n \delta)^{\frac{1}{4(N+m)}}.
\]
As a result, there exists a (deterministic) universal constant $C$ such that  
\[
\mathbb E_U \left[ \mathbb E\left[ \mathds{1}_{|T_n(U)|\le 2 \delta} \mathds{1}_{\mathcal A} \big| \mathcal G\right]\right]=\mathbb E \left[ \mathbb E_U\left[ \mathds{1}_{|T_n(U)|\le 2 \delta} \mathds{1}_{\mathcal A} \right]\big| \mathcal G\right]\leq C  (n \delta)^{\frac{1}{4(N+m)}}.
\]
In view of Equation \eqref{eq.pu2}, combining  this last estimate with the one obtained in \eqref{eq:Ac}, we conclude that 
\[
\mathbb P_U \otimes \mathbb P\left(\left|S_n(U)\right|\le \delta \right) \leq \left( C  (n \delta)^{\frac{1}{4(N+m)}}+ \kappa^{\lfloor\frac{N}{m+1}\rfloor}\right)^{1/2},
\]
hence the theorem.
\end{proof}

Applying Theorem \ref{thm:anti} with  $\delta_n=\frac{1}{\lfloor n^{\beta} \rfloor !}$, $N_n=n^{\alpha}$, $m_n=n^{\gamma}$ with $0<\gamma<\alpha<\beta<1$, say $\alpha=(\beta+\gamma)/2$ for simplicity, we  obtain the following analogue of Corollary \ref{cor.anti}.
  \begin{corollary}\label{cor.anti:gen}
Let $(a_{k,n})_{k \geq 1, n \geq 1}$ and $(b_{k,n})_{k \geq 1, n \geq 1}$ be two independent arrays of $m_n-$dependent random variables, satisfying Conditions (A.1)--(A.3), with $m_n=O(n^{\gamma})$ and $\gamma<1$ . Then, for any $1>\beta>\gamma$, there exists a constant $c>0$ such that for $n$ large enough and for all realizations $X(x)$ of the variable $X$, there exists $t=t(x) \in [0, 2\pi]$ satisfying
 	\[
 	\mathbb P\left(|S_n(t)|\le \frac{1}{\lfloor n^{\beta} \rfloor !} \right) \leq e^{-cn^{\frac{\beta-\gamma}{2}}}.
 	\]
 	In particular, for $n$ large enough, we have
 	\[
 	\mathbb P \otimes \mathbb P_X\left(||S_n ||_{\infty} \le \frac{1}{\lfloor n^{\beta} \rfloor !} \right) \leq e^{-cn^{\frac{\beta-\gamma}{2}}}.
 	\]
 \end{corollary}

\section{Uniform integrability and universality} \label{sec.ui}

The goal of this section is to derive uniform integrability estimates for $\mathcal N(S_n,[0,2\pi])$. Namely, we will show that that this sequence of random variables is bounded in $L^{1+\epsilon}(\mathbb P \otimes \mathbb P_X)$ for some $\epsilon>0$.

\begin{prop}\label{pro.ui}
For any fixed small $\epsilon>0$, there exists an exponent $\gamma_{\epsilon}< \frac{\eta}{2(1+\eta)}$ such that if $m_n=O(n^{\gamma_{\epsilon}})$, we have 
\[
\sup_{n \geq 1} \mathbb E \, \mathbb E_X \left[ |\mathcal N(S_n,[0,2\pi])|^{1+\epsilon}\right] <+\infty.
\]
\end{prop}
\noindent
Combined with Corollary \ref{cor.conv.nodal}, the last Proposition \ref{pro.ui} ensures that 
\[
\lim_{n \to +\infty} \mathbb E \, \mathbb E_X \left[ |\mathcal N(S_n,[0,2\pi])|\right] = \mathbb E \, \mathbb E_X \left[ |\mathcal N(S_{\infty},[0,2\pi])|\right].
\]
Again, under Condition \eqref{eq.lowerb} on the positiveness of the spectral density, the zeros of the limit process $S_{\infty}$ identify with the zeros of the $\sin_c$ Gaussian process $Z$, for which it is well know that 
\[
\mathbb E \, \mathbb E_X \left[ |\mathcal N(Z,[0,2\pi])|\right] = \frac{2}{\sqrt{3}}.
\]
As a result, thanks to the representation formula \eqref{eq.repzero}, we have indeed 
\[
\lim_{n \to +\infty} \frac{\mathbb E[\mathcal N(f_n, [0,2\pi])]}{n} = \frac{2}{\sqrt{3}}.
\]
The rest of the section is dedicated to the proof of Proposition \ref{pro.ui}. 

\subsection{Reduction of the problem}\label{sec.reduc}
Let us fix a small $\epsilon > 0$ and write
\[
\mathbb E \, \mathbb E_X \left[ |\mathcal N(S_n,[0,2\pi])|^{1+\epsilon}\right]  = (1+\epsilon) \int_0^{+\infty} s^{\epsilon} \, \mathbb P \otimes \mathbb P_X \left(\mathcal N(S_n,[0,2\pi])>s  \right)ds.
\]
Since the number of zeros of $S_n$ is bounded by $2n$, we have in fact 
\[
\mathbb E \, \mathbb E_X \left[ |\mathcal N(S_n,[0,2\pi])|^{1+\epsilon}\right]  = (1+\epsilon)  \int_0^{2n} s^{\epsilon} \, \mathbb P \otimes \mathbb P_X \left(\mathcal N(S_n,[0,2\pi])>s  \right)ds.
\]
By iterating Rolle Lemma $\lfloor s \rfloor$ times, see for example p.19 of \cite{MR4491425}, the probability in the integrand can then be upper bounded as follows
\[
\mathbb P \otimes \mathbb P_X \left(\mathcal N(S_n,[0,2\pi])>s  \right)  \leq \mathbb P \otimes \mathbb P_X \left(||S_n||_{\infty} \leq \frac{(2\pi)^{\lfloor s \rfloor}}{\lfloor s \rfloor !}||S_n^{(\lfloor s \rfloor)}||_{\infty}\right),
\]
so that for any $R>0$
\begin{equation}\label{eq.bornprob0}
\begin{array}{ll}
\mathbb P \otimes \mathbb P_X \left(\mathcal N(S_n,[0,2\pi])>s  \right)   \leq & \mathbb P \otimes \mathbb P_X \left(||S_n||_{\infty} \leq \frac{(2\pi)^{\lfloor s \rfloor}R}{\lfloor s \rfloor !} \right) \\
& + \mathbb P \otimes \mathbb P_X \left( ||S_n^{(\lfloor s \rfloor)}||_{\infty}>R\right).
\end{array}
\end{equation}
Applying Markov inequality, we get
\[
\mathbb P \otimes\mathbb P_X \left( ||S_n^{(\lfloor s \rfloor)}||_{\infty}>R\right) \leq \frac{1}{R^2} \times \mathbb E \, \mathbb E_X \left[ ||S_n^{(\lfloor s \rfloor)}||_{\infty}^2 \right].
\]
Now, comparing the uniform norm with Sobolev norms, see Lemma 5.15 p.107 of \cite{MR2424078}, if $|| \cdot ||_{2}$ denotes the standard $L^2$ norm in $L^2([0,2\pi], \mathbb P_X)$, there exists a universal constant $C>0$ such that 
\[
\mathbb E_X \left[ ||S_n^{(\lfloor s \rfloor)}||_{\infty}^2 \right]  \leq C \left(\mathbb E_X \left[   ||S_n^{(\lfloor s \rfloor)}||_{2}^2\right] +\mathbb E_X \left[   ||S_n^{(\lfloor s \rfloor+1)}||_{2}^2 \right] \right).
\]
But for any integer $\ell$, we have 
\[
\mathbb E_X \left[   ||S_n^{( \ell )}||_{2}^2\right] = \frac{1}{2n} \sum_{k=1}^n \left(\frac{k}{n}\right)^{2\ell} \left( a_k^2+b_k^2\right) \leq \frac{1}{2n} \sum_{k=1}^n \left( a_k^2+b_k^2\right), 
\]
so that using Condition (A.1)
\[
\mathbb P \otimes \mathbb P_X \left( ||S_n^{(\lfloor s \rfloor)}||_{\infty}>R\right) \leq  \frac{2C}{R^2}.
\]
Let us now choose $R=R(s)$ of the form $R(s):=(1+|s|)^{1/2+\epsilon}$ so that
\begin{equation}\label{eq.bornprob1}
\sup_{n \geq 1}\mathbb P \otimes \mathbb P_X \left( ||S_n^{(\lfloor s \rfloor)}||_{\infty}>R(s)\right) \leq \frac{2C}{(1+|s|)^{1+2\epsilon}}.
\end{equation}
We have then 
 \[
\sup_{n \geq 1} \int_0^{2n} s^{\epsilon} \, \mathbb P \otimes\mathbb P_X \left( ||S_n^{(\lfloor s \rfloor)}||_{\infty}>R(s)\right) ds \leq \int_0^{+\infty} \frac{2C s^{\epsilon}}{(1+|s|)^{1+2\epsilon}}ds<+\infty.
\]
In view of Equation \eqref{eq.bornprob0}, in order to establish Proposition \ref{pro.ui}, we are left to show that 
\[
\sup_{n \geq 1} \int_0^{2n} s^{\epsilon} \, \mathbb P \otimes \mathbb P_X \left(||S_n||_{\infty} \leq \frac{(2\pi)^{\lfloor s \rfloor}R(s)}{\lfloor s \rfloor !} \right)ds<+\infty.
\]
We will decompose the last integral as a sum of two terms, whose asymptotics behaviors will be discussed in the next Sections \ref{sec.using.rate.clt} and \ref{sec.using.smallball} respectively. 
Recall that Proposition \ref{pro.rate} and Corollary \ref{cor.rate} are valid as soon as $m_n = O(n^{\gamma})$, with $\gamma<\frac{\eta}{2(1+\eta)}$. Let us then define 
\[
\gamma_{\epsilon}:=\frac{\eta}{2(5 +7 \eta + 3\epsilon (4 +6 \eta))}, \quad \text{so that} \quad  (1+3\epsilon) \gamma_{\epsilon}=\left( \frac{1+\eta}{4+6\eta}\right) \left( \frac{\eta}{2(1+\eta)}-\gamma_{\epsilon} \right)
\]
which is the power that appeared in Corollary \ref{cor.rate} that we will use shortly. We then set 
\[
\beta_{\epsilon}:=\frac{1+3\epsilon}{1+2\epsilon} \gamma_{\epsilon},\quad \text{so that} \quad \beta_{\epsilon}>\gamma_{\epsilon}
\]
and 
\[
(1+\epsilon) \beta_{\epsilon} < (1+2\epsilon) \beta_{\epsilon}=(1+3\epsilon) \gamma_{\epsilon}=\left( \frac{1+\eta}{4+6\eta}\right) \left( \frac{\eta}{2(1+\eta)}-\gamma_{\epsilon} \right).
\]
Finally, set
\[
I_n:=\int_0^{n^{\beta_{\epsilon}}} s^{\epsilon} \, \mathbb P \otimes \mathbb P_X \left(||S_n||_{\infty} \leq \frac{(2\pi)^{\lfloor s \rfloor}R(s)}{\lfloor s \rfloor !} \right)ds,
\]
\[
J_n:=\int_{n^{\beta_{\epsilon}}}^{2n} s^{\epsilon} \, \mathbb P \otimes \mathbb P_X \left(||S_n||_{\infty} \leq \frac{(2\pi)^{\lfloor s \rfloor}R(s)}{\lfloor s \rfloor !} \right)ds.
\]
\subsection{Using the rate of convergence in the CLT} \label{sec.using.rate.clt}
First replacing the sup norm by the evaluation at zero, we have 
\[
I_n \leq \int_0^{n^{\beta_{\epsilon}}} s^{\epsilon} \, \mathbb P \otimes \mathbb P_X \left(|S_n(0)| \leq \frac{(2\pi)^{\lfloor s \rfloor}R(s)}{\lfloor s \rfloor !} \right)ds.
\]
Now using the rate of convergence in Central limit Theorem given by Proposition \ref{pro.rate}, and more precisely the upper bound given by Corollary \ref{cor.rate}, we get
\[
I_n = O\left( \int_0^{n^{\beta_{\epsilon}}} s^{\epsilon} \left(   \frac{2(2\pi)^{\lfloor s \rfloor}R(s)}{\pi \sqrt{\kappa_{\rho}}\lfloor s \rfloor !} \wedge 1 \right) ds\right) + O\left( n^{\beta_{\epsilon}(1+\epsilon)- \left( \frac{1+\eta}{4+6\eta}\right) \left( \frac{\eta}{2(1+\eta)}-\gamma_{\epsilon} \right)}\right).
\]
Our choice of $(\beta_{\epsilon}, \gamma_{\epsilon})$ ensures that the last term is negligible and the above integral is convergent as $n$ goes to infinity, therefore
\[
\sup_{n \geq 1}I_n <+\infty.
\]
\subsection{Using the small-ball estimate} \label{sec.using.smallball}
To upper bound the integral $J_n$, we will use the anti-concentration estimate established in Corollary \ref{cor.anti}. Namely, for $n$ sufficiently large and for $s \geq n^{\beta_{\epsilon}}$, we have for all $\gamma_{\epsilon}<\beta<\beta_{\epsilon}$,
\[
 \frac{(2\pi)^{\lfloor s \rfloor}R(s)}{\lfloor s \rfloor !} \leq \frac{1}{\lfloor  n^{\beta} \rfloor !}
\]
so that by Corollary \ref{cor.anti}, we get
\[
\mathbb P \otimes \mathbb P_X \left(||S_n||_{\infty} \leq \frac{(2\pi)^{\lfloor s \rfloor}R(s)}{\lfloor s \rfloor !} \right) = O \left( e^{-cn^{\frac{\beta-\gamma_{\epsilon}}{2}}}\right).
\]
As a result, we conclude that as $n$ goes to infinity
\[
J_n =O\left( e^{-cn^{\frac{\beta-\gamma_{\epsilon}}{2}}} \int_{n^{\beta_{\epsilon}}}^{2n} s^{\epsilon} ds\right) =o(1), 
\]
and in particular 
\[
\sup_{n \geq 1} J_n <+\infty.
\]

\begin{remark}
Since the exponent $\epsilon$ can be chosen arbitrarily small to ensure the uniform integrability of the sequence $(\mathcal N(S_n,[0,2\pi]))$, the proof above shows that the exponent $\gamma$ can be chosen arbitrarily close to the threshold $\gamma_0:=\frac{\eta}{2(5 +7 \eta)}$. In particular, if the random coefficients have uniformly bounded moments of all orders, letting $\eta$ go to infinity, the exponent $\gamma$ can be chosen arbitrarily close to $1/14$.
\end{remark}

\section{Extension to functionals of Gaussian processes}\label{sec.gauss}
This last section is devoted to the proof of Theorem \ref{theo.main.gauss} in the framework where the coefficients of the random polynomials are functions of a stationary Gaussian process. We thus suppose here that $a_k=H(X_k)$ where $X_k$ is a weakly stationary Gaussian process with covariance function $\rho_G$ and spectral density $\psi_G$, and $H$ has a finite $(2+\eta)$ moment with respect to the standard Gaussian measure. By assumption, the spectral density is assumed to be smooth and positive on the whole period $[-\pi, \pi]$. Note in particular that the hypothesis that $\psi_G \in \mathcal C^{\infty}$ ensures that $\rho_G(|k|) = O(\frac{1}{k^D})$ for any given exponent $D>1$, so that the spectral density is the limit of its Fourier sum
\[
\psi_{G}(x) =\sum_{k\in \mathbb Z} \rho_G(|k|) e^{ikx}.
\]
Let us consider a sequence $(m_n)$ of integers going to infinity with $n$ and such that we have $m_n=O(n^{\gamma})$ for $\gamma<\gamma_{0}$ and let us introduce the truncated function 
\[
\rho_{G,n}(|k|) =\rho_G(k) \mathds{1}_{|k|\leq m_n}.
\]
Then, if we set  
\[
\psi_{{G,n}}(x) =\sum_{k\in \mathbb Z} \rho_{G,n}(|k|) e^{ikx},
\] 
we have for any fixed $D$
\begin{equation}\label{eq.approx.uni}
||\psi_{{G,n}}-\psi_{G}||_{\infty} \leq \sum_{|k|>m_n}\rho_G(k)= O\left( \frac{1}{m_n^{D-1}} \right) .
\end{equation}
In particular, since $\kappa_G:=\inf \psi_G>0$ by assumption, we deduce that for $n$ large enough, we have also
\begin{equation}\label{eq.lowerb2}
\inf\{ \psi_{G,n}(x); x \in [-\pi, \pi]\}>0.
\end{equation}
As a result, for all $(z_i)_{1\leq i \leq m_n} \in \mathbb C^{m_n}$
\[
\sum_{k,\ell=1}^{m_n} \rho_{G,n}(k-\ell) z_k \bar{z}_{\ell} =\int_{-\pi}^{\pi} \left| \sum_{k=1}^{m_n} z_k e^{i k x} \right|^2\psi_{G,n}(x)dx>0
\]
so that the truncated function $\rho_{G,n}$ is also a covariance function. 

\subsection{Total variation distance between normal vectors}
Let us denote by $\Sigma_{G}$ and $\widetilde{\Sigma}_{G}$ the covariance matrices of size $n$
\[
\Sigma_{G}(i,j) = \rho_G(|i-j|), \quad\widetilde{\Sigma}_{G}(i,j) = \rho_{G,n}(|i-j|), \quad 1\leq i,j\leq n.
\]
Following Theorem 1.1 of \cite{devroye2023total}, if $\mathrm{d}_{TV}$ denote the total variation distance, we have 
\[
\mathrm{d}_{TV} \left(\mathcal N_n\left(0, \Sigma_{G} \right), \mathcal N_n\left(0, \widetilde{\Sigma}_{G} \right)\right) \leq \frac{3}{2} \min\left( 1, \sqrt{\sum_{i=1}^n \lambda_i^2 }\right) 
\]	
where $\lambda_i$ are the eigenvalues of the symmetric matrix $\Sigma_{G}^{-1} \widetilde{\Sigma}_{G}-\mathrm{Id}$. Now if $\text{sp}$ denote the spectral radius, we have 
\[
\begin{array}{ll}
\displaystyle{\mathrm{Trace}\left((\Sigma_{G}^{-1} \widetilde{\Sigma}_{G}-\mathrm{Id})^2\right)} & \leq \displaystyle{n \times \mathrm{sp}^2\left(\Sigma_{G}^{-1} \widetilde{\Sigma}_{G}-\mathrm{Id}\right)}\\
\\
& \leq \displaystyle{n \times \mathrm{sp}^2\left(\Sigma_{G}^{-1}\right)\mathrm{sp}^2\left( \widetilde{\Sigma}_{G}-\Sigma_G\right)}\\
\\
& \leq  \displaystyle{n \times \frac{1}{\kappa_G^2} \times \sum_{|i-j|>m_n} \rho^2(|i-j|)}\\
\\
& \leq  \displaystyle{n \times \frac{1}{\kappa_G^2} \times n \sum_{k>m_n} \rho^2(|k|)}.
\end{array}
\]
For the third inequality, we used the fact that $\text{sp}(\Sigma_G^{-1}) \geq \frac{1}{\kappa_{G}}$, which is a consequence of Szeg\"o Theorem on Toeplitz matrices, see Chapter 5 of \cite{MR0094840}.  As a result, we get
\[
\mathrm{d}_{TV} \left(\mathcal N_n\left(0, \Sigma_{G} \right), \mathcal N_n\left(0, \widetilde{\Sigma}_{G} \right)\right)=O \left(  \frac{n}{m_n^{D-1/2}}\right)= O \left(  \frac{1}{n^{\gamma (D-1/2)-1}}\right).
\]
In particular, for $D$ large enough so that $\gamma(D-1/2)-1>0$, the total variation distance goes to zero as $n$ goes to infinity.

\subsection{Convergence of the process with truncated covariance}
So let us now consider independent copies $(X_k)_{k\geq 1}$ and $(Y_k)_{k\geq 1}$ of a Gaussian process with covariance $\rho_G$ and two independent copies $(\widetilde{X}_{k,n})_{1\leq k\leq n}$ and $(\widetilde{Y}_{k,n})_{1\leq k\leq n}$ of a centered Gaussian vector 
with covariance matrix $\widetilde{\Sigma}_G$. We then set, for $t\in [0,2\pi]$
\[
\begin{array}{ll}
S_n(t)& :=\displaystyle{\frac{1}{\sqrt{n}}\sum_{k=1}^n H(X_k) \cos\left(k X+\frac{kt}{n}\right)+H(Y_k) \sin\left(k X+\frac{kt}{n}\right)},\\
\\
\widetilde{S}_n(t)& :=\displaystyle{\frac{1}{\sqrt{n}}\sum_{k=1}^n H(\widetilde{X}_{k,n}) \cos\left(k X+\frac{kt}{n}\right)+H(\widetilde{Y}_{k,n}) \sin\left(k X+\frac{kt}{n}\right).}
\end{array}
\]
Recall the relation between the spectral density $\psi_{G,n}$ of the Gaussian process $(\widetilde{X}_{k,n})$ and the spectral density $\psi_{\rho, n}$ of the process $H(\widetilde{X}_{k,n})$
\[
\psi_{\rho,n}(x) = \sum_q c_q^2 q! \psi_{G,n}^{\ast q}(x),
\]
where $(c_q)$ are the Hermite coefficients of the function $H$. Combining Equations \eqref{eq.approx.uni} and \eqref{eq.lowerb2}, as $n$ goes to infinity, we have 
\[
|| \psi_{\rho,n}-\psi_{\rho}||_{\infty} =o(1), \quad \text{where}\quad \psi_{\rho}(x) := \sum_q c_q^2 q! \psi_{G}^{\ast q}(x).
\]
Indeed, by Equation \eqref{eq.approx.uni}, for any given integer $Q>0$, we have 
\[
\limsup_{n \to +\infty} \left|\left| \sum_{|q| \leq Q} c_q^2 q! \psi_{G,n}^{\ast q} - \sum_{|q| \leq Q} c_q^2 q! \psi_{G}^{\ast q}\right|\right|_{\infty}=0.
\]
Otherwise, since the spectral density $\psi_{G}$ is smooth, it is uniformly bounded on the interval $[-\pi, \pi]$ and using Equation \eqref{eq.approx.uni} again, so is $\psi_{G,n}$. In particular for $n$ large enough we have  $||\psi_{G,n}||_{\infty} \leq 2||\psi_{G}||_{\infty}$ and similarly $||\psi_{G,n}^{\ast q} ||_{\infty} \leq 2||\psi_{G}^{\ast q}||_{\infty}\leq 2||\psi_{G}^{}||_{\infty}$ for all $q\geq 1$. Therefore, for $n$ large enough, we have 
\[
\left|\left| \sum_{|q| > Q} c_q^2 q! \psi_{G,n}^{\ast q} - \sum_{|q| >Q} c_q^2 q! \psi_{G}^{\ast q}\right|\right|_{\infty} \leq 3 ||\psi_{G}||_{\infty} \sum_{|q|>Q} c_q^2 q! ,
\]
and the rest of the series goes to zero as $Q$ goes to infinity since the function $H$ is square integrable with respect to the standard Gaussian measure. 
Moreover, since $\kappa_G=\inf \psi_{G}>0$ by hypothesis, we deduce that $\kappa_{\rho}:=\inf \psi_{\rho}>0$.
In other words, the ``truncated'' process $(\widetilde{S}_n(t))_{t \in [0,2\pi]}$ satisfies the hypotheses of Theorem \ref{theo.clt}. As a result, the latter converges in the $\mathcal C^1$ topology towards an explicit limit process, say $S_{\infty}$ with a stochastic representation
\[
S_{\infty} = \sqrt{2\pi \psi_{\rho}(X)} Z,
\]
with $Z$ a Gaussian process with $\sin_c$ covariance, independent of $X$.

\subsection{Comparison of distributions}The random variables $S_n(t)$ and $\widetilde{S}_n(t)$ being measurable functions of the coefficients $(X_k,Y_k)_{k\geq 1}$ and $(\widetilde{X}_{k,n},\widetilde{Y}_{k,n})_{1\leq k\leq n}$ respectively, we have 
\begin{equation}\label{eq.dist.marg}
\mathrm{d}_{TV} \left(S_n(t), \widetilde{S}_n(t)\right) \leq \mathrm{d}_{TV} \left(\mathcal N_n\left(0, \Sigma_{G} \right), \mathcal N_n\left(0, \widetilde{\Sigma}_{G} \right)\right) = O \left(  \frac{1}{n^{\gamma(D-1/2)-1}}\right) .
\end{equation}
In the same manner, for any $d \geq 1$ and $(t_1, \ldots, t_d)$ and $(\xi_1, \ldots, \xi_d)$, any linear combinations will satisfy
\begin{equation}\label{eq.dist.combi}
\mathrm{d}_{TV} \left(\sum_{k=1}^d \xi_k S_n(t_k), \sum_{k=1}^d \xi_k \widetilde{S}_n(t_k)\right) = O \left(  \frac{1}{n^{\gamma(D-1/2)-1}}\right) .
\end{equation}
The sup norm $||S_n|_{\infty}$ being also a measurable function of the coefficients,  we even have
\begin{equation}\label{eq.dist_uni}
\mathrm{d}_{TV} \left(||S_n||_{\infty}, ||\widetilde{S}_n||_{\infty}\right) \leq \mathrm{d}_{TV} \left(\mathcal N_n\left(0, \Sigma_{G} \right), \mathcal N_n\left(0, \widetilde{\Sigma}_{G} \right)\right) = O \left(  \frac{1}{n^{\gamma(D-1/2)-1}}\right) .
\end{equation}
Combined with the tightness estimates of Section \ref{sec.tight}, Equation \eqref{eq.dist.combi} shows that the non truncated process $(S_n(t))_{t \in [0,2\pi]}$ converges in the $\mathcal C^1$ topology towards the same limit as $(\widetilde{S}_n(t))_{t \in [0,2\pi]}$. 
Again combining the Continuous Mapping Theorem and the continuity of the ``number of zeros'' functional with respect to the $\mathcal C^1$ topology, we deduce that the number of $\mathcal N(S_n, [0,2\pi])$ converges in distribution to $\mathcal N(S_{\infty}, [0,2\pi])$.
\par
\medskip
\noindent
Using Equation \eqref{eq.dist.marg} with $t=0$, in combination with Corollary \ref{cor.rate} applied to $\widetilde{S}_n(0)$, we deduce that uniformly in $s>0$ and for $n$ large enough, we have moreover 
\begin{equation}\label{eq.rate2}
 \mathbb P \otimes \mathbb P_X \left(|S_n(0)|<s  \right) \leq \frac{2s }{\pi\sqrt{\kappa_{\rho}} }\wedge 1 +O\left(n^{-\left( \frac{1+\eta}{4+6\eta}\right) \left( \frac{\eta}{2(1+\eta)}-\gamma \right)} \right)+ O \left(  \frac{1}{n^{\gamma(D-1/2)-1}}\right).
\end{equation}
Last, combining Equation \eqref{eq.dist_uni} with Corollary \ref{cor.anti:gen}, we get that for $0<\gamma<\beta<1$
\begin{equation}\label{eq.anti2:gen}
\mathbb P \otimes \mathbb P_X\left(||S_n ||_{\infty} \le \frac{1}{\lfloor n^{\beta} \rfloor !} \right) \leq e^{-cn^{\frac{\beta-\gamma}{2}}}+O \left(  \frac{1}{n^{\gamma(D-1/2)-1}}\right).
\end{equation}
\subsection{Uniform integrability and universality}
Let us complete the proof of Theorem \ref{theo.main.gauss}. We are left to prove that the sequence $(\mathcal N(S_n, [0,2\pi]))_{n \geq 1}$ is uniformly integrable and possibly bounded in $L^{1+\epsilon}(\mathbb P \otimes \mathbb P_X)$ for some small positive exponent $\epsilon$. Recall that we have chosen $\gamma<\gamma_0$ here, so that using the same notations as in Section \ref{sec.reduc}, we can choose $\epsilon$ small enough and work with $\gamma=\gamma_{\epsilon}<\gamma_0$ and the associated exponent $\beta_{\epsilon}$. We can then follow the exact same strategy as adopted in Section \ref{sec.ui} with the only differences that 
\begin{itemize}
\item via Equation \eqref{eq.rate2}, the upper bound for term $I_n$ in Section \ref{sec.using.rate.clt} is replaced by 
\[
\begin{array}{ll}
I_n  \leq & \displaystyle{O\left( \int_0^{n^{\beta_{\epsilon}}} s^{\epsilon} \left(   \frac{2(2\pi)^{\lfloor s \rfloor}R(s)}{\pi \sqrt{\kappa_{\rho}}\lfloor s \rfloor !} \wedge 1 \right) ds\right) + O\left(n^{-\left( \frac{1+\eta}{4+6\eta}\right) \left( \frac{\eta}{2(1+\eta)}-\gamma_{\epsilon} \right)} \right)}\\
\\ & + \displaystyle{O \left(  \frac{n^{\beta_{\epsilon}(1+\epsilon)}}{n^{\gamma_{\epsilon}(D-1/2)-1}}\right)},
\end{array}
\]
\item via Equation \eqref{eq.anti2:gen}, the upper bound for term $J_n$ in Section \ref{sec.using.smallball} is replaced by 
\[
J_n =O\left( e^{-cn^{\frac{\beta-\gamma_{\epsilon}}{2}}}  \int_{n^{\beta_{\epsilon}}}^{2n} s^{\epsilon} ds\right) +O\left( \frac{n^{1+\epsilon}}{n^{\gamma(D-1/2)-1}} \right).
\]
\end{itemize}
Since the covariance function $\rho_G$ is supposed to have fast decay, one can choose $D$ sufficiently large so that as $n$ goes to infinity
\[
O \left(  \frac{n^{\beta_{\epsilon}(1+\epsilon)}}{n^{\gamma_{\epsilon}(D-1/2)-1}}\right) =o(1), \quad O\left( \frac{n^{1+\epsilon}}{n^{\gamma_{\epsilon}(D-1/2)-1}} \right)=o(1).
\]
Since $\beta<1$, the second term above is the largest. Therefore, since $\epsilon$ can be chosen arbitrary small, the critical index of regularity $D_{0}$ satisfies the equation  $2=\gamma_{0}(D_{0}-1/2)$, which leads to $D_{0}=57/2+20/\eta$. 
\bibliographystyle{alpha}
\bibliography{newbib}

\begin{thebibliography}{DMR23}

\bibitem[ADP19]{MR3876743}
J\"{u}rgen Angst, Federico Dalmao, and Guillaume Poly.
\newblock On the real zeros of random trigonometric polynomials with dependent
  coefficients.
\newblock {\em Proc. Amer. Math. Soc.}, 147(1):205--214, 2019.

\bibitem[AF03]{MR2424078}
Robert~A. Adams and John J.~F. Fournier.
\newblock {\em Sobolev spaces}, volume 140 of {\em Pure and Applied Mathematics
  (Amsterdam)}.
\newblock Elsevier/Academic Press, Amsterdam, second edition, 2003.

\bibitem[AP21]{MR4350983}
J\"{u}rgen Angst and Guillaume Poly.
\newblock Variations on {S}alem-{Z}ygmund results for random trigonometric
  polynomials: application to almost sure nodal asymptotics.
\newblock {\em Electron. J. Probab.}, 26:Paper No. 156, 36, 2021.

\bibitem[APP22]{MR4491425}
J\"{u}rgen Angst, Thibault Pautrel, and Guillaume Poly.
\newblock Real zeros of random trigonometric polynomials with dependent
  coefficients.
\newblock {\em Trans. Amer. Math. Soc.}, 375(10):7209--7260, 2022.

\bibitem[BCP19]{MR3980307}
Vlad Bally, Lucia Caramellino, and Guillaume Poly.
\newblock Non universality for the variance of the number of real roots of
  random trigonometric polynomials.
\newblock {\em Probab. Theory Related Fields}, 174(3-4):887--927, 2019.

\bibitem[DMR23]{devroye2023total}
Luc Devroye, Abbas Mehrabian, and Tommy Reddad.
\newblock The total variation distance between high-dimensional gaussians with
  the same mean, 2023.

\bibitem[DNN22]{MR4452640}
Yen Do, Hoi~H. Nguyen, and Oanh Nguyen.
\newblock Random trigonometric polynomials: universality and non-universality
  of the variance for the number of real roots.
\newblock {\em Ann. Inst. Henri Poincar\'{e} Probab. Stat.}, 58(3):1460--1504,
  2022.

\bibitem[DNV18]{MR3846831}
Yen Do, Oanh Nguyen, and Van Vu.
\newblock Roots of random polynomials with coefficients of polynomial growth.
\newblock {\em Ann. Probab.}, 46(5):2407--2494, 2018.

\bibitem[Erd45]{erdos1945lemma}
Paul Erd\H{o}s.
\newblock On a lemma of {L}ittlewood and {O}fford.
\newblock 1945.

\bibitem[Fla17]{MR3718101}
Hendrik Flasche.
\newblock Expected number of real roots of random trigonometric polynomials.
\newblock {\em Stochastic Process. Appl.}, 127(12):3928--3942, 2017.

\bibitem[GS58]{MR0094840}
Ulf Grenander and Gabor Szeg\"{o}.
\newblock {\em Toeplitz forms and their applications}.
\newblock California Monographs in Mathematical Sciences. University of
  California Press, Berkeley-Los Angeles, 1958.

\bibitem[IKM16]{MR3563891}
Alexander Iksanov, Zakhar Kabluchko, and Alexander Marynych.
\newblock Local universality for real roots of random trigonometric
  polynomials.
\newblock {\em Electron. J. Probab.}, 21:Paper No. 63, 19, 2016.

\bibitem[Jan21]{janson2021central}
Svante Janson.
\newblock A central limit theorem for m-dependent variables, 2021.

\bibitem[JPR24]{MR4710537}
Svante Janson, Luca Pratelli, and Pietro Rigo.
\newblock Quantitative bounds in the central limit theorem for {$m$}-dependent
  random variables.
\newblock {\em ALEA Lat. Am. J. Probab. Math. Stat.}, 21(1):245--265, 2024.

\bibitem[Kri16]{krishnapur2016anti}
Manjunath Krishnapur.
\newblock Anti-concentration inequalities.
\newblock {\em Lecture notes}, page~1, 2016.

\bibitem[LA22]{liu2022wasserstein}
Tianle Liu and Morgane Austern.
\newblock Wasserstein-p bounds in the central limit theorem under weak
  dependence.
\newblock {\em arXiv preprint arXiv:2209.09377}, 2022.

\bibitem[LA23]{liu2023smooth}
Tianle Liu and Morgane Austern.
\newblock Smooth edgeworth expansion and wasserstein-$ p $ bounds for mixing
  random fields.
\newblock {\em arXiv preprint arXiv:2309.07031}, 2023.

\bibitem[Naz93]{MR1246419}
F.~L. Nazarov.
\newblock Local estimates for exponential polynomials and their applications to
  inequalities of the uncertainty principle type.
\newblock {\em Algebra i Analiz}, 5(4):3--66, 1993.

\bibitem[NV13]{nguyen2013small}
Hoi~H Nguyen and Van~H Vu.
\newblock Small ball probability, inverse theorems, and applications.
\newblock {\em Erd{\H{o}}s centennial}, pages 409--463, 2013.

\bibitem[NV21]{MR4340724}
Oanh Nguyen and Van Vu.
\newblock Random polynomials: central limit theorems for the real roots.
\newblock {\em Duke Math. J.}, 170(17):3745--3813, 2021.

\bibitem[RS01]{MR1908331}
Alexander Rusakov and Oleg Seleznjev.
\newblock On weak convergence of functionals on smooth random functions.
\newblock {\em Math. Commun.}, 6(2):123--134, 2001.

\end{thebibliography}
\end{document}